\documentclass{amsart}
\usepackage{amsmath,amssymb,amscd}
\usepackage{color}
\usepackage{mathtools}
\usepackage{lineno}
\usepackage{bigints}
\usepackage{comment}
\usepackage{mathdots}

\begin{document}

\makeatletter
\def\section{\@startsection {section}{1}{\z@}{-3.5ex plus -1ex minus -.2ex}{2.3 ex plus .2ex}{\large\bf}}
\numberwithin{equation}{subsection}
\makeatother

\theoremstyle{definition}
\newtheorem{dfn}{Definition}[subsection]
\newtheorem{thm}[dfn]{Theorem}
\newtheorem*{th*}{Theorem}
\newtheorem{lem}[dfn]{Lemma}
\newtheorem{prop}[dfn]{Proposition}
\newtheorem{rem}[dfn]{Remark}
\newtheorem{cor}[dfn]{Corollary}
\newtheorem*{cor*}{Corollary}
\newtheorem*{prop*}{Proposition}
\newtheorem{quest}[dfn]{Question}
\newtheorem{conj}[dfn]{Conjecture}

\newcommand{\bbA}{\mathbb{A}}
\newcommand{\bbC}{\mathbb{C}}
\newcommand{\bbR}{\mathbb{R}}
\newcommand{\bbQ}{\mathbb{Q}}
\newcommand{\bbZ}{\mathbb{Z}}
\newcommand{\bbH}{\mathbb{H}}

\newcommand{\calA}{\mathcal{A}}
\newcommand{\calC}{\mathcal{C}}
\newcommand{\calD}{\mathcal{D}}
\newcommand{\calE}{\mathcal{E}}
\newcommand{\calF}{\mathcal{F}}
\newcommand{\calH}{\mathcal{H}}
\newcommand{\calI}{\mathcal{I}}
\newcommand{\calN}{\mathcal{N}}
\newcommand{\calO}{\mathcal{O}}
\newcommand{\calP}{\mathcal{P}}
\newcommand{\calS}{\mathcal{S}}
\newcommand{\calU}{\mathcal{U}}
\newcommand{\calV}{\mathcal{V}}
\newcommand{\calZ}{\mathcal{Z}}

\newcommand{\fraka}{\mathfrak{a}}
\newcommand{\frakA}{\mathfrak{A}}
\newcommand{\frakg}{\mathfrak{g}}
\newcommand{\frakH}{\mathfrak{H}}
\newcommand{\frakh}{\mathfrak{h}}
\newcommand{\fraki}{\mathfrak{i}}
\newcommand{\frakk}{\mathfrak{k}}
\newcommand{\frakl}{\mathfrak{l}}
\newcommand{\frakn}{\mathfrak{n}}
\newcommand{\frakN}{\mathfrak{N}}
\newcommand{\frakp}{\mathfrak{p}}
\newcommand{\frakX}{\mathfrak{X}}

\newcommand{\bfs}{\mathbf{s}}
\newcommand{\bfe}{\mathbf{e}}
\newcommand{\bfa}{\mathbf{a}}

\newcommand{\rank}{\mathop{\mathrm{rank}}}
\newcommand{\corank}{\mathop{\mathrm{Corank}}}
\newcommand{\im}{\mathop{\mathrm{Im}}}
\newcommand{\Hom}{\mathop{\mathrm{Hom}}}
\newcommand{\op}{\mathrm{op}}

\newcommand{\GL}{\mathrm{GL}}
\newcommand{\Sym}{\mathrm{Sym}}
\newcommand{\Sp}{\mathrm{Sp}}
\newcommand{\Mat}{\mathrm{Mat}}
\newcommand{\Mp}{\mathrm{Mp}}
\newcommand{\SL}{\mathrm{SL}}

\newcommand{\automforms}{\mathcal{A}(\Gamma)}
\newcommand{\NHMFonG}{\mathcal{N}(\Gamma)}
\newcommand{\NHMFonGwtlam}{\mathcal{N}_\lambda(\Gamma)}
\newcommand{\gk}{(\frakg, K_\infty)}
\newcommand{\NHMFonGcharchi}{\mathcal{N}(\Gamma,\chi)}
\newcommand{\NHAFonPG}{\mathcal{N}(P \backslash G)}
\newcommand{\NHAFonG}{\mathcal{N}(G)}
\newcommand{\NHAFonBG}{\mathcal{N}(B \backslash G)}

\newcommand{\ul}{\underline}
\newcommand{\pr}{\mathop{\mathrm{pr}}}
\newcommand{\Ext}{\mathop{\mathrm{Ext}}}
\newcommand{\fin}{{\rm \mathchar`- fin}}
\newcommand{\fini}{\mathrm{fin}}
\newcommand{\Ind}{\mathrm{Ind}}
\newcommand{\ind}{\mathrm{ind}}
\newcommand{\Lie}{\mathrm{Lie}}
\newcommand{\Res}{\mathop{\mathrm{Res}}}
\newcommand{\bs}{\backslash}
\newcommand{\tr}{\mathrm{tr}}
\newcommand{\bole}{\bold{e}}
\newcommand{\supp}{\mathop{\mathrm{supp}}}
\newcommand{\sgn}{\mathrm{sgn}}
\newcommand{\bfi}{\mathbf{i}}
\newcommand{\Wh}{\mathrm{Wh}}
\newcommand{\cusp}{\mathrm{cusp}}
\newcommand{\bfm}{\mathbf{m}}
\newcommand{\Aut}{\mathrm{Aut}}
\newcommand{\unit}{\mathrm{unit}}
\newcommand{\End}{\mathrm{End}}
\newcommand{\vep}{\varepsilon}
\newcommand{\abcd}{\begin{pmatrix}a & b \\ c & d \end{pmatrix}
}


\title{On the classification of $(\mathfrak{g},K)$-modules generated by nearly holomorphic Hilbert-Siegel modular forms and projection operators}
\author{Shuji Horinaga}

\begin{abstract}
	We classify the $(\mathfrak{g},K)$-modules generated by nearly holomorphic Hilbert-Siegel modular forms by the global method.
	As an application, we study the image of projection operators on the space of nearly holomorphic Hilbert-Siegel modular forms with respect to infinitesimal characters in terms of $(\frakg,K)$-modules.
\end{abstract}

\maketitle
\markboth{$(\frakg,K)$-modules generated by nearly holomorphic modular forms}{$(\frakg,K)$-modules generated by nearly holomorphic modular forms}

\section{Introduction}
	
	\subsection{Algebraicity of special $L$ values}
	The arithmeticity of special $L$ values is a central problem in modern number theory.
	In the motivic setting, Deligne \cite{Del} conjectured the algebraicity of critical $L$ values up to the period.
	For the critical values attached to scalar valued Hilbert-Siegel modular forms and Hermitian modular forms, Shimura proved the arithmeticity of them up to suitable periods in \cite{00_Shimura} by using of nearly holomorphic modular forms.
	The period can be expressed by Petersson inner product times some power of $\pi$.
	Recently, in \cite{HPSS}, Pitale, Saha, Schmidt and the author prove the arithmeticity of them attached to vector valued Siegel modular forms under the parity condition of weights.
	The purpose of this paper is to prepare to remove the parity condition by investigating the $(\frakg,K)$-modules generated by nearly holomorphic Hilbert-Siegel modular forms.
	
	\subsection{$(\frakg,K)$-modules generated by nearly holomorphic Siegel modular forms}
	Let $F$ be a totally real field with degree $d$ and $\bfa$ the set of embeddings of $F$ into $\bbR$.
	Put $G_n = \Res_{F/\bbQ}\Sp_{2n}$.
	Here $\Res$ is the Weil restriction and $\Sp_{2n}$ is the symplectic group of rank $n$.
	Let $\frakH_n$ be the Siegel upper half space of degree $n$.
	Put $\frakg_n = \Lie(G_n(\bbR)) \otimes _\bbR \bbC$.
	We denote by $K_{n,\infty}$ and $\calZ_n$ the stabilizer of $\mathbf{i} = (\sqrt{-1}\,\mathbf{1}_n,\ldots,\sqrt{-1}\,\mathbf{1}_n) \in \frakH_n^d$ and the center of the universal enveloping algebra $\calU(\frakg_n)$, respectively.
	Let $K_{n,\bbC}$ be the complexification of $K_{n,\infty}$.
	Set $\frakk_n = \Lie(K_{n,\infty}) \otimes_\bbR\bbC$.
	We then have the well-known decomposition:
	\[
	\frakg_n = \frakk_n \oplus \frakp_{n,+} \oplus \frakp_{n,-}.
	\]
	Here $\frakp_{n,+}$ (resp.~$\frakp_{n,-}$) is the Lie subalgebra of $\frakg_n$ corresponding to the holomorphic tangent space (resp.~anti-holomorphic tangent space) of $\frakH_n^d$ at $\mathbf{i}$.
	We take a Cartan subalgebra of $\frakk_n$.
	Then it is a Cartan subalgebra of $\frakg_n$.
	The root system $\Phi$ of $\mathfrak{sp}_{2n}(\bbC)$ is 
	\[
	\Phi = \{ \; \pm (e_i + e_j), \; \pm(e_k - e_\ell), \quad 1\leq i \leq j \leq n, 1 \leq k < \ell \leq n \; \}.
	\] 
	We consider the set
	\[
	\Phi^+ = \{ \; - (e_i + e_j), \; e_k - e_\ell, \quad 1\leq i \leq j \leq n, 1 \leq k < \ell \leq n \; \}
	\]
	to be a positive root system.
	Let $\rho$ be half the sum of positive roots.
	Note that $\frakg_n = \bigoplus_{v \in \bfa} \mathfrak{sp}_{2n}(\bbC)$.
	We say that a weight $\lambda = (\lambda_{1,v},\ldots,\lambda_{n,v})_{v\in \bfa}$ which lies in $\bigoplus_{v \in \bfa}\bbC^{n}$ is $\frakk_n$-dominant if $\lambda_{i,v}-\lambda_{i+1,v} \in \bbZ_{\geq 0}$ for any $1 \leq i \leq n-1$ and $v \in \bfa$.
	We also say that a $\frakk_n$-dominant integral weight $\lambda = (\lambda_{1,v},\ldots,\lambda_{n,v})_{v\in \bfa}$ is anti-dominant if $\lambda_n \geq n$.
	For any $\frakk_n$-dominant integral weight $\lambda$, there exist the (parabolic) Verma module $N(\lambda)$ with respect to a parabolic subalgebra $\frakp_{n,-} \oplus \frakk_n$ and a unique irreducible highest weight $(\frakg_n,K_{n,\infty})$-module $L(\lambda)$ of highest weight $\lambda$.
	Then, $L(\lambda)$ is the unique irreducible quotient of $N(\lambda)$.
	For a $(\frakg_n, K_{n,\infty})$-module $\pi$, the symbol $\pi^\vee$ denotes the contragredient of $\pi$ in the sense of \cite{cat_o}.
	
	For an automorphic form $\varphi$ on $G_n(\bbA_\bbQ)$, we say that $\varphi$ is nearly holomorphic if $\varphi$ is $\frakp_{n,-}$-finite, i.e., $\calU(\frakp_{n,-}) \cdot \varphi$ is finite-dimensional.
	The goal of this paper is to classify the indecomposable $(\frakg_n,K_{n,\infty})$-modules generated by nearly holomorphic automorphic forms.
	
	\begin{thm}[Theorem \ref{classification_th}]
	Let $\pi$ be an indecomposable $(\frakg_n,K_{n,\infty})$-module generated by a nearly holomorphic automorphic form on $G_n(\bbA_\bbQ)$.
	If $F \neq \bbQ$, $\pi$ is irreducible.
	If $F = \bbQ$, the length of $\pi$ is at most two.
	More precisely, if $\pi$ is reducible, there exists an odd integer $i$ and $(\lambda_1, \ldots, \lambda_{n-i}) \in \bbZ^{n-i}$ with $\lambda_1 \geq \cdots \geq \lambda_{n-i} \geq n-(i-3)/2$ such that $\pi \cong N(\lambda_1,\ldots,\lambda_{n-i},n-(i-3)/2,\ldots,n-(i-3)/2)^\vee$.
	\end{thm}
	
	This result is a generalization of \cite{PSS_21}.
	The key idea of proof is the harmonic analysis of the space of nearly holomorphic automorphic forms on $G_n(\bbA_\bbQ)$, which is investigated in \cite{Horinaga_2}.
	
	\subsection{Projection operators}
	Fix a weight $\rho$ and a congruence subgroup $\Gamma$.
	Let $N_\rho(\Gamma)$ be the space of nearly holomorphic Hilbert-Siegel modular forms of weight $\rho$ with respect to $\Gamma$.
	For an infinitesimal character $\chi$ of $\calZ_n$, we can define the projection operator $\frakp_{\chi} \in \mathrm{End}(N_\rho(\Gamma))$ associated to $\chi$.
	Then, the projection operator $\frakp_{\chi}$ commutes with the $\Aut(\bbC)$ action as follows:
	
	\begin{thm}[Theorem \ref{proj_comm}]
	For any $f \in N_\rho(\Gamma)$ and $\sigma \in \Aut(\bbC)$, we have
	\[
	\frakp_{\chi}({^\sigma f}) = {^\sigma \frakp_\chi(f)}.
	\]
	\end{thm}
	
	For a $\frakk_n$-dominant integral weight $\lambda$ and $v \in \bfa$, put $j_v(\lambda) = \#\{j \mid \lambda_{1,v} \equiv \lambda_{j,v} \, (\mathrm{mod}\, 2)\}$.
	Set
	\[
	\wedge^{j_v(\lambda)} = \wedge^{j_v(\lambda)}\mathrm{std}_{\GL_n(\bbC)}, \quad \rho_v = \mathrm{det}^{\lambda_{1,v}-1} \otimes \wedge^{j_v}\mathrm{std}_{\GL_n(\bbC)}, \quad \text{and} \quad  \rho = \bigotimes_{v\in\bfa} \rho_v,
	\]
	where $\mathrm{std}_{\GL_n(\bbC)}$ is the standard representation of $\GL_n(\bbC)$ and $\wedge^{j_v(\lambda)}\mathrm{std}_{\GL_n(\bbC)}$ is the $j_v(\lambda)$-th exterior product of $\mathrm{std}_{\GL_n(\bbC)}$.
	
	\begin{thm}[Theorem \ref{proj}]
	Let $\lambda = (\lambda_{1,v},\ldots,\lambda_{n,v})_v$ be a regular anti-dominant integral weight.
	Put $\rho = \bigotimes_{v \in \bfa} (\det^{\lambda_{1,v}-1} \otimes \wedge^{j_v(\lambda)})$ and $N_\rho(\Gamma,\chi_\lambda) = \frakp_{\chi_\lambda}(N_\rho(\Gamma))$.
	If $F=\bbQ$ and $\lambda_{n,v} = n+1$, any modular form in $N_\rho(\Gamma,\chi_\lambda)$ generates $L(\lambda)$ or $N(\lambda_1,\ldots,\lambda_{n-1},n-1)^\vee$.
	If not, any modular form in $N_\rho(\Gamma,\chi_\lambda)$ generates $L(\lambda)$.
	\end{thm}
	
	The following is the analogue of holomorphic projection.
	
	\begin{cor}
	Let $\lambda = (\lambda_{1,v},\ldots,\lambda_{n,v})_v$ be an anti-dominant $\frakk_n$-dominant integral weight and $\rho$ the irreducible highest weight representation of $K_{n,\bbC}$ with highest weight $\lambda$.
	Suppose $\lambda_{1,v}-\lambda_{n,v} \leq 1$ and $\lambda_{n,v} \geq n+1$ for any $v \in \bfa$.
	If $F \neq \bbQ$ or $\lambda_{n,v} \neq n+1$ for some $v \in \bfa$, the projection $\frakp_\chi$ defines a projection onto $M_\rho(\Gamma)$, the subspace of holomorphic modular forms.
	\end{cor}
	
	We then characterize the nearly holomorphic Hilbert-Siegel modular forms which generate a holomorphic discrete series representation in terms of projections $\frakp_\chi$ under a mild assumption.
	This gives a generalization of Shimura's holomorphic projection.
	
	\subsection*{Notation}
	
	We denote by $\mathrm{Mat}_{m,n}$ the set of $m \times n$-matrices.
	Put $\Mat_{n} = \Mat_{n,n}$ with the unit $\mathbf{1}_n$.
	Let $\GL_n$ and $\Sp_{2n}$ be the algebraic groups defined by
	\[
	\GL_n(R) = \{g \in \Mat_n \mid \det g \in R^\times \}
	\]
	and
	\[
	\Sp_{2n} (R) = \left\{ g \in \GL_{2n}(R) \, \middle| \, {^t{g}}\begin{pmatrix}0_n & -\mathbf{1}_n \\ \mathbf{1}_n & 0_n \end{pmatrix} g = \begin{pmatrix}0_n & -\mathbf{1}_n \\ \mathbf{1}_n & 0_n \end{pmatrix}\right\}
	\]
	for a ring $R$, respectively.
	Set $\Sym_n = \{g \in \Mat_n \mid {^tg}=g\}$.
	Let $B_n$ be the subgroup of $\Sp_{2n}$ defined by
	\[
	B_n = \left\{ \begin{pmatrix}a & * \\ 0 & {^ta^{-1}} \end{pmatrix} \,\middle|\, \text{$a$ is a upper triangular matrix.}\right\}.
	\]
	The group $B_n$ is a Borel subgroup of $\Sp_{2n}$ with the Levi decomposition $B_n = T_nN_n$.
	Here $T_n \subset B_n$ is the maximal diagonal torus of $\Sp_{2n}$.
	A parabolic subgroup $P$ of $\Sp_{2n}$ is called standard if $P$ contains $B_n$.
	Let $A_P$ be the split component of $P$ and $A_P^\infty$ the identity component of $A_P(\bbR)$.
	We denote by $P_{i,n}$ and $Q_{i,n}$ the standard parabolic groups of $\Sp_{2n}$ with the Levi subgroups $\GL_i \times \Sp_{2(n-i)}$ and $(\GL_1)^i \times \Sp_{2(n-i)}$, respectively.
	Set $P_n = P_{n,n}$.
	For a parabolic subgroup $P$, let $\delta_P$ be the modulus character of $P$.

	For $n \in \bbZ_{\geq 1}$, set
	\[
	\frakH_n = \{z \in \Sym_n(\bbC) \mid \text{$\mathrm{Im}(z)$ is positive definite}\}.
	\]
	The space $\frakH_n$ is called the Siegel upper half space of degree $n$.
	The Lie group $\Sp_{2n}(\bbR)$ acts on $\frakH_n$ by the rule
	\[
	\begin{pmatrix}a & b \\ c& d\end{pmatrix}(z) = (az+b)(cz+d)^{-1} , \qquad \begin{pmatrix}a & b \\ c& d\end{pmatrix} \in \Sp_{2n}(\bbR), \, z \in \frakH_n.
	\]
	Put
	\[
	K_{n,\infty} = \left\{g = \begin{pmatrix} a & b \\ c & d \end{pmatrix} \in \Sp_{2n}(\bbR) \, \middle| \, a = d, \, c = -b\right\}.
	\]
	Then $K_{n,\infty}$ is the group of stabilizers of $\bfi = \sqrt{-1} \, \mathbf{1}_n \in \frakH_n$.
	For simplicity the notation, the symbol $\bfi$ also denotes the element $(\sqrt{-1}\,\mathbf{1}_n, \ldots, \sqrt{-1} \, \mathbf{1}_n) \in \frakH_n^d$.
	Since the action of $\Sp_{2n}(\bbR)$ on $\frakH_n$ is transitive, we have $\frakH_n \cong \Sp_{2n}(\bbR)/ K_{n,\infty}$.
	
	Let $F$ be a totally real field with degree $d$.
	Let $\bfa = \{\infty_1,\ldots,\infty_d\}$ be the set of embeddings of $F$ into $\bbR$.
	We denote by $\bbA_F$ and $\bbA_{F,\fini}$ the adele ring of $F$ and the finite part of $\bbA_F$, respectively.
	For a place $v$, let $F_v$ be the $v$-completion of $F$.
	Put $F_\infty = \prod_{v \in \bfa} F_v$.
	For a non-archimedean place $v$, let $\calO_{F_v}$ be the ring of integers of $F_v$.

	Set $G_n = \Res_{F/\bbQ}\Sp_{2n}$ where $\Res$ is the Weil restriction.
	We define the standard parabolic subgroups $P_{i,n}, Q_{i,n}$ and $B_n$ of $G_n$ by the Weil restriction of parabolic subgroups $P_{i,n}, Q_{i,n}$ and $B_n$ of $\Sp_{2n}$, respectively.
	Let $W_n$ be the Weyl group of $\Sp_{2n}$.
	For an archimedean place $v$, set $K_{n,v} = K_{n,\infty}$.
	For the sake of simplicity, the symbol $K_{n,\infty}$ denotes the maximal compact subgroup $\prod_{v \in \bfa}K_{n,v}$ of $G_n(\bbR)$.
	Let $K_{n,\bbC}$ be the complexification of $\prod_{v \in \bfa} K_{n,v}$.
	Put $\frakg_n = \mathrm{Lie}(G_n(\bbR)) \otimes_\bbR \bbC$ and $\frakk_n = \mathrm{Lie}(\prod_{v \in \bfa}K_{n,v})\otimes_\bbR\bbC$.
	Set $K_v = \Sp_{2n}(\calO_{F_v})$ for a non-archimedean place $v$.
	We denote by $\calZ_n$ the center of the universal enveloping algebra $\calU(\frakg_n)$.
	We then obtain the well-known decomposition
	\[
	\frakg_n =  \frakk_n \oplus \frakp_{n,+} \oplus \frakp_{n,-}
	\]
	where $\frakp_{n,+}$ (resp.~$\frakp_{n,-}$) is the Lie subalgebra of $\frakg_n$ corresponding to the holomorphic tangent space (resp.~anti-holomorphic tangent space) of $\frakH_n^d$ at $\bfi$.
	It is well-known that the Lie algebras $\frakg_n$ and $\frakk_n$ have the same Cartan subalgebra.
	We fix such a Cartan subalgebra.
	Then the root system of $\mathfrak{sp}_{2n}(\bbC)$ is 
	\[
	\Phi = \{ \; \pm (e_i + e_j), \; \pm(e_k - e_\ell), \quad 1\leq i \leq j \leq n, 1 \leq k < \ell \leq n \; \}.
	\]
	We consider the set
	\begin{align*}
	\Phi^+ = \{ \; - (e_i + e_j), \; e_k - e_\ell, \quad 1\leq i \leq j \leq n, 1 \leq k < \ell \leq n \; \}
	\end{align*}
	to be a positive root system.
	Let $\rho$ be half the sum of positive roots.
	Put $\rho_{i,n} = n - (i-1)/2$ and $\rho_n = \rho_{n,n}$.
	This corresponds to half the sum of roots in the unipotent subgroup of $P_{i,n}$.
	For $\lambda = (\lambda_{1,v},\ldots,\lambda_{n,v})_v \in \bigoplus_{v \in \bfa} \bbC^n$, we say that $\lambda$ is a weight if $\lambda_{i,v}-\lambda_{i+1,v} \in \bbZ$ for any $v$ and $1 \leq i \leq n-1$.
	The weight $\lambda$ is $\frakk_n$-dominant if $\lambda_{i,v}-\lambda_{i+1,v} \geq 0$ for any $v$ and $1 \leq i \leq n-1$.
	For a $\frakk_n$-dominant weight $\lambda$, let $\rho_\lambda$ be an irreducible finite-dimensional representation of $\frakk_n$.
	When $\lambda$ is integral, i.e., any entry of $\lambda$ is an integer, we identify $\rho_\lambda$ as the derivative of an irreducible finite-dimensional representation of $K_{n,\bbC}$ with highest weight $\lambda$.
	We then write the representation of $K_{n,\bbC}$ by the same $\rho_{\lambda}$.
	
	We fix a non-trivial additive character $\psi = \bigotimes_v \psi_v$ of $F \bs \bbA_F$ as follows:
	If $F=\bbQ$, let
	\begin{align*}
	\psi_p(x) &= \exp(-2\pi\sqrt{-1}\, y), \qquad x \in \bbQ_p,  \\
	\psi_\infty(x) &= \exp(2\pi\sqrt{-1} \, x),\qquad x \in \bbR, 
	\end{align*}
	where $y \in \cup_{m=1}^\infty p^{-m}\bbZ$ such that $x-y \in \bbZ_p$.
	In general, for an archimedean place $v$ of $F$, put $\psi_v = \psi_\infty$ and for a non-archimedean place $v$ with the rational prime $p$ divisible by $v$, put $\psi_v(x) = \psi_p(\mathrm{Tr}_{F_v/\bbQ_p}(x))$.
	
	For a function $f$ on a group $G$, let $r$ be the right translation, i.e., $r(g)f(h) = f(hg)$ for any $g,h \in G$.
	For a subset $H$ of $G$, we denote by $f|_H$ the restriction of $f$ to $H$.
	Let $G$ be a Lie group with the Lie algebra $\frakg$.
	For a smooth function $f$ on $G$ and $X \in \frakg$, put
	\[
	X\cdot f(g) = \left.\frac{d}{dt} \right|_{t=0}f(g\exp(tX)), \qquad g \in G.
	\]
	For the action of $G_n(\bbA_\bbQ)$, we mean the $G_n(\bbA_{\bbQ,\fini}) \times (\frakg_n,K_{n,\infty})$-action.

	\section{Nearly holomorphic Hilbert-Siegel modular forms and automorphic forms}
	
	In this section, we review the definition and arithmeticity of nearly holomorphic Hilbert-Siegel modular forms.
	We also recall some properties of nearly holomorphic automorphic forms on $G_n(\bbA_\bbQ)$ and basic terminologies of automorphic forms.
	
	\subsection{Differential operators on the Siegel upper half space}
	
	We recall the differential operators on $\frakH_n$.
	For details, see \cite[\S 12]{00_Shimura}.
	Fix a basis on $\Sym_n(\bbC)$ by $\{(1+\delta_{i,j})^{-1}(e_{i,j}+e_{j,i}) \mid 1 \leq i \leq j \leq n\}$.
	We denote the basis by $\{\vep_\nu\}$.
	For $u \in \Sym_n(\bbC)$, write $u = \sum_\nu u_\nu \vep_\nu$ with $u_\nu \in \bbC$ and for $z \in \frakH_n$, write $z = \sum_\nu z_\nu \vep_\nu$ with $z_\nu \in \bbC$.
	For a non-negative integer $e$ and a finite-dimensional vector space $V$, let $S_e(\Sym_n(\bbC),V)$ be the space of $V$-valued homogeneous polynomial maps of degree $e$ on $\Sym_n(\bbC)$ and $\mathrm{Ml}_e(\Sym_n(\bbC),V)$ the space of $e$-multilinear maps on $\Sym_n(\bbC)^e$ to $V$.
	Note that $S_e(\Sym_n(\bbC),V)$ can be viewed as the space of symmetric elements of $\mathrm{Ml}_e(\Sym_n(\bbC),V)$.
	For a representation $\rho$ of $\GL_n(\bbC)$ on $V$, we define representations $\rho \otimes \tau^e$ and $\rho\otimes\sigma^e$ on $\mathrm{Ml}_e(\Sym_n(\bbC),V)$ by
	\[
	((\rho \otimes \tau^e)(a)h)(u_1,\ldots,u_e) = \rho(a)h({^ta}u_1a,\ldots,{^ta}u_ea)
	\]
	and
	\[
	((\rho \otimes \sigma^e)(a)h)(u_1,\ldots,u_e) = \rho(a)h({a}^{-1}u_1{^ta^{-1}},\ldots,{a}^{-1}u_e{^ta^{-1}}),
	\]
	respectively.
	Here, $h \in \mathrm{Ml}_e(\Sym_n(\bbC),V)$, $a \in \GL_n(\bbC)$ and $(u_1,\ldots,u_e) \in \Sym_n(\bbC)^e$.
	The symbols $\rho \otimes \tau^e$ and $\rho \otimes \sigma^e$ also denote the restrictions to the representations space $S_e(\Sym_n(\bbC),V)$.
	
	For $f \in C^\infty(\frakH_n,V)$, we define functions $Df,\overline{D}f, Cf, E,f$ on $C^\infty(\frakH_n,S_1(\Sym_n(\bbC), V))$ by
	\begin{align*}
	((Df)(z))(u) = \sum_\nu u_\nu \frac{\partial f}{\partial z_\nu}(z)&, \qquad ((\overline{D}f)(z))(u) = \sum_\nu u_\nu \frac{\partial f}{\partial \overline{z_\nu}}(z),\\
	((Cf)(z))(u) = 4 ((Df)(z))(yuy)&,\qquad ((Ef)(z))(u) = 4((\overline{D}f)(z))(yuy).
	\end{align*}
	Here, $u = \sum_\nu u_\nu \vep_\nu \in \Sym_n(\bbC), z = \sum_\nu z_\nu\vep_\nu \in \frakH_n$ and $y = \mathrm{Im}(z)$.
	For $f \in C^\infty(\frakH_n,V)$, we say that $f$ is nearly holomorphic if there exists $e$ such that $E^ef = 0$.

	\subsection{Definition}
		
	Let $F$ be the fixed totally real field.
	For an integral ideal $\frakn$ of $F$, set
	\[
	\Gamma(\frakn) = \left\{ \gamma \in \Sp_{2n}(\calO_F)\, \middle|\, \gamma - \mathbf{1}_{2n} \in \Mat_{2n}(\frakn)\right\}.
	\]
	The group $\Gamma(\frakn)$ is called the principal congruence subgroup of $G_{n}(\bbQ)$ of level $\frakn$.
	We say that a subgroup $\Gamma$ of $G_{n}(\bbQ)$ is a congruence subgroup if there exists an integral ideal $\frakn$ such that $\Gamma$ contains $\Gamma(\frakn)$ and $[\Gamma \colon \Gamma(\frakn)] < \infty$.
	In this subsection, we regard $G_n(\bbQ)$ as a subgroup of $G_n(\bbR) = \prod_{v \in \bfa} \Sp_{2n}(F_v)$ by $\gamma \longmapsto (\infty_1(\gamma), \ldots, \infty_d(\gamma))$.
	Similarly, we regard a congruence subgroup $\Gamma$ of $G_{n}(\bbQ)$ as a subgroup of $G_{n}(\bbR)$.
	
	We define the factor of automorphy $j \colon G_n(\bbR) \times \frakH_n^d \longrightarrow \GL_n(\bbC)^d$ by
	\[
	j(g,z) = (c_vz_v+d_v)_v \in \prod_{v \in \bfa}\GL_n(\bbC) = \GL_n(\bbC)^d, \qquad g = \left(\begin{pmatrix}a_v&b_v\\c_v&d_v\end{pmatrix}\right)_v \in G_n(\bbR), \quad z = (z_v)_v \in \frakH_n^d.
	\]
	For a representation $\rho$ of $K_{n,\bbC}$ on $V$, set $j_\rho = \rho \circ j$.
	For $g \in G_n(\bbR)$, we define the slash operator $|_\rho g$ on $C^\infty(\frakH_n^d,V)$ by
	\[
	(f|_\rho g) (z_1,\ldots,z_d) = j_\rho(g, z)^{-1} f(\gamma(z_1,\ldots,z_d)),
	\]
	for $f \in C^\infty(\frakH_n^d , V)$ and $(z_1,\ldots,z_d) \in \frakH^d_n$.
	Let $\Gamma$ be a congruence subgroup of $G_n(\bbQ)$.
	Suppose that a function $f \in C^\infty(\frakH_n^d, V)$ satisfies the automorphy $f|_\rho\gamma = f$ for any $\gamma \in \Gamma$.
	Then, $f$ has the Fourier expansion
	\[
	(f|_{\rho}\gamma)(z) = \sum_{h \in \Sym_n(F)} c_f(h,y,\gamma) \mathbf{e}(\tr({hz})), \qquad z \in \frakH_n^d, \, y = \mathrm{Im}(z)
	\]
	where $\mathbf{e}(\tr(hz)) = \exp(2\pi\sqrt{-1}\,\sum_{j=1}^h \tr(\infty_j(h)z_j))$ for $(z_1,\ldots,z_d) \in \frakH_n^d$ and $h \in \Sym_n(F)$.
	We consider the following condition:
	If $c_f(h,y,\gamma) \neq 0$, the matrix $h$ is positive semi-definite.
	We call this condition the cusp condition.
	We say that a $V$-valued $C^\infty$-function $f$ on $\frakH_n^d$ is a nearly holomorphic Hilbert-Siegel modular form of weight $\rho$ with respect to $\Gamma$ if $f$ satisfies the following conditions:
		\begin{itemize}
			\item $f$ is a nearly holomorphic function.
			\item $f|_\rho \gamma = f$ for all $\gamma \in \Gamma$.
			\item $f$ satisfies the cusp condition.
		\end{itemize}

	We denote by $N_\rho(\Gamma)$ the space of nearly holomorphic Hilbert-Siegel modular forms of weight $\rho$ with respect to $\Gamma$.
	In the following, for modular forms, we mean a (nearly holomorphic) Hilbert-Siegel modular forms.
	By K\"oecher principle, we can remove the cusp condition if $n >1$ or $F \neq \bbQ$. 
	For the proof, see \cite[Proposition 4.1]{RC_Horinaga} for $n>1$.
	We can give the same proof for the case of $F \neq \bbQ$.
	For simplicity, if $\rho = \det^k$, we say that a modular form of weight $\det^k$ is a modular form of weight $k$.
	
	\subsection{$\Aut(\bbC)$ action for nearly holomorphic Hilbert-Siegel modular forms and the holomorphic projection}
	
	Let $f$ be a nearly holomorphic modular form of weight $\rho$ with respect to $\Gamma$.
	Take a model $V$ of $\rho$ and fix a rational structure of $V$.
	Then, Shimura introduced the $\Aut(\bbC)$-action on $f$.
	For details, see \cite[\S 14.11]{00_Shimura} and \cite[\S 3.3]{HPSS}.
	For $\sigma \in \Aut(\bbC)$, we denote by ${^\sigma f}$ the action of $\sigma$ on $f$.
	For a weight $\rho = \bigotimes_{v\in\bfa} \rho_v$, put ${^\sigma \rho} = \bigotimes_{v\in\bfa} \rho_{\sigma \circ v}$.
	The following theorem is proved in \cite[Theorem 14.12]{00_Shimura}.
	
	\begin{thm}
	For $f \in N_{\rho}(\Gamma)$ and $\sigma \in \Aut(\bbC)$, one has ${^\sigma f} \in N_{^\sigma \rho}(\Gamma)$.
	\end{thm}
	
	Let $M_\rho(\Gamma)$ be the space of holomorphic functions in $N_\rho(\Gamma)$.
	Set $N^p_\rho(\Gamma) = N_\rho^{(p_v)_v}(\Gamma) = \{f \in N_\rho(\Gamma) \mid \text{$E_v^{p_v+1}f = 0$ for any $v \in \bfa$}\}$.
	The, $N^0_\rho(\Gamma) = M_\rho(\Gamma)$.
	Let $\rho = \bigotimes_v \rho_v$ be a character of $K_{n,\bbC}$ with the weight $(k_v)_{v \in \bfa}$.
	Take non-negative integers $p_v$ satisfies $k_v > n+p_v$ or $k_v < n+(3-p_v)/2$ for any $v \in \bfa$.
	Put $p = (p_v)_v$.
	Then, in \cite[\S 15.3]{00_Shimura}, Shimura introduced a projection $\frakA \colon N_\rho^p(\Gamma) \longrightarrow M_\rho(\Gamma)$.
	The projection $\frakA$ is called the holomorphic projection.
	By Shimura \cite[Proposition 15.3]{00_Shimura}, it commutes with the $\Aut(\bbC)$ actions as follows:
	
	\begin{thm}
	With the above notation, for any $\sigma \in \Aut(\bbC)$ and $f \in N_\rho(\Gamma)$, one has $\frakA({^\sigma f}) = {^\sigma \frakA(f)}$.
	\end{thm}
	
%
	
	In \cite[\S 3.4]{HPSS}, we define other projection operators $\frakp_\chi$ associated to infinitesimal characters $\chi$ of $\calZ_n$.
	This can be viewed as a generalization of the holomorphic projection $\frakA$.
	In this paper, we study the image of $\frakp_{\chi}$ in terms of $(\frakg_n,K_{n,\infty})$-modules.
	
	\subsection{Automorphic forms on $G_n(\bbA_\bbQ)$}
	
	Let $P=MN$ be a standard parabolic subgroup of $G_n$.
	For a smooth function $\phi: N(\bbA_\bbQ)M(\bbQ) \backslash G_n(\bbA_\bbQ) \longrightarrow \bbC$, we say that $\phi$ is automorphic if it satisfies the following conditions:
	\begin{itemize}
	\item $\phi$ is right $K_n$-finite.
	\item $\phi$ is $\mathcal{Z}_n$-finite.
	\item $\phi$ is slowly increasing.
	\end{itemize}
	We denote by $\mathcal{A}(P \backslash G_n)$ the space of automorphic forms on $N(\bbA_\bbQ)M(\bbQ) \bs G_n(\bbA_\bbQ)$.
	For simplicity, we write $\calA(G_n)$ when $P=G_n$.
	The space $\mathcal{A}(P \backslash G_n)$ is stable under the action of $G_n(\bbA_\bbQ)$.
	
	For parabolic subgroups $P$ and $Q$ of $G_n$, we say that $P$ and $Q$ are associate if the split components $A_P$ and $A_Q$ are $G_n(\bbQ)$-conjugate. 
	We denote by $\{P\}$ the associated class of the parabolic subgroup $P$.
	For a locally integrable function $\varphi$ on $N_P(\bbQ) \bs G_n(\bbA_\bbQ)$, set
	\[
	\varphi_{P}(g) = \int_{N_P(\bbQ) \bs N_P(\bbA_\bbQ)} \varphi(ng) \, dn
	\]
	where $P = M_PN_P$ is the Levi decomposition of $P$ and the Haar measure $dn$ is normalized by
	\[
	\int_{N_P(\bbQ) \bs N_P(\bbA_\bbQ)} \, dn = 1.
	\]
	The function $\varphi_P$ is called the constant term of $\varphi$ along $P$.
	If $\varphi$ lies in $\calA(P \bs G_n)$, $\varphi_Q$ is an automorphic form on $N_Q(\bbA_\bbQ)M_Q(\bbQ) \bs G_n(\bbA_\bbQ)$ for a parabolic subgroup $Q \subset P$.
	We call $\varphi$ cuspidal if $\varphi_Q$ is zero for any standard parabolic subgroup $Q$ of $G$ with $Q \subsetneq P$.
	We denote by $\calA_{\cusp}(P \bs G_n)$ the space of cusp forms in $\calA(P \bs G_n)$.
	For a character $\xi$ of the split component $A_P^\infty$, put
	\[
	\calA(P \bs G)_\xi = \{\varphi \in \calA(P \bs G_n) \mid \text{$\varphi(ag) = a^{\xi + \rho_P} \varphi(g)$ for any $g \in G_n(\bbA_\bbQ)$ and $a \in A_P^\infty$}\}.
	\]
	Here, $\rho_P$ is the character of $A_P^\infty$ corresponding to half the sum of roots of $N_P$ relative to $A_P$.
	We define $\calA_\cusp(P \bs G)_\xi$ similarly.
	Set
	\[
	\calA(P \bs G_n)_Z = \bigoplus_\xi \calA(P \bs G_n)_\xi, \qquad \calA_\cusp(P \bs G_n)_Z = \bigoplus_\xi \calA_\cusp(P \bs G_n)_\xi.
	\]
	Here, $\xi$ runs over all the characters of $A_P^\infty$.
	Let $\fraka_P$ be the real vector space generated by coroots associated to the root system of $G_n$ relative to $A_P$.
	Then, by \cite[Lemma I.3.2]{MW}, there exist canonical isomorphisms
	\begin{align}\label{finite_function}
	\bbC[\fraka_P] \otimes \calA(P \bs G_n)_Z \cong \calA(P \bs G), \qquad \bbC[\fraka_P] \otimes \calA_\cusp(P \bs G_n)_Z \cong \calA_\cusp(P \bs G_n).
	\end{align}
	
	For a standard Levi subgroup $M$, set
	\[
	M(\bbA_\bbQ)^1 = \bigcap_{\chi \in \Hom_{\mathrm{conti}}(M(\bbA_\bbQ),\bbC^\times)} \mathrm{Ker}(|\chi|).
	\]
	For a function $f$ on $G_n(\bbA_\bbQ)$ and $g \in G_n(\bbA_\bbQ)$, let $f_g$ be the function on $M_P(\bbA_\bbQ)^1$ defined by $m \longmapsto m^{-\rho_P}f(mg)$.
	Put
	\[
	\calA(G_n)_{\{P\}} = \left\{\varphi \in \calA(G) \,\middle|\, \begin{matrix}\text{$\varphi_{Q,ak}$ is orthogonal to all cusp forms on $M_Q(\bbA_Q)^1$}\\ \text{for any $a \in A_Q, k\in K_n$, and $Q \not \in \{P\}$}\end{matrix}\right\}.
	\] 
	By \cite[Lemma I.3.4]{MW}, $\calA(G_n)_{\{G\}}$ is equal to $\calA_\cusp(G_n)$.
	More precisely, Langlands \cite{langlands} had proven the following result:
	\begin{thm}
	With the above notation, we have
	\[
	\calA(G_n) = \bigoplus_{\{P\}}\calA(G_n)_{\{P\}},
	\]
	where $\{P\}$ runs through all associated classes of parabolic subgroups.
	\end{thm}
	
	Let $M$ be a standard Levi subgroup of $G_n$ and $\tau$ an irreducible cuspidal automorphic representation of $M(\bbA_\bbQ)$.
	We say that a cuspidal datum is a pair $(M, \tau)$ such that $M$ is a Levi subgroup of $G_n$ and that $\tau$ is an irreducible cuspidal automorphic representation of $M(\bbA_\bbQ)$.
	Take $w \in W_n$.
	Put $M^w=wMw^{-1}$ and let $P^w=M^wN^w$ be the standard parabolic subgroup with Levi subgroup $M^w$.
	The irreducible cuspidal automorphic representation $\tau^w$ of $M^w(\bbA_\bbQ)$ is defined by $\tau^w(m')=\tau(w^{-1}m'w)$ for $m'\in M^w(\bbA_\bbQ)$.
	Two cuspidal data $(M, \tau)$ and $(M', \tau')$ are called equivalent if there exists $w\in W(M)$ such that $M' = M^w$ and that $\tau'=\tau^w$.
	Here we put
	\[
	W(M) = \left\{ w \in W \,\middle|\, 
	\begin{matrix}
	\text{$w M w^{-1}$ is a standard Levi subgroup of $G_n$}\\
	\text{and $w$ has a minimal length in $wW_M$}
	\end{matrix}
	\right\}
	\]
	where $W_M$ is the Weyl group of $M$.
	
	Let $\mathcal{A}(G_n)_{(M, \tau)}$ is the subspace of automorphic forms in $\mathcal{A}(G_n)$ with the cuspidal support $(M, \tau)$.
	For the definition, see \cite[\S III.2.6]{MW}.
	Then the following result is well-known.
	For example, see \cite[Theorem III.2.6]{MW}.
	
	\begin{thm}\label{cusp_supp_decomp}
	The space $\mathcal{A}(G_n)$ is decomposed as
	\[
	\mathcal{A}(G_n)=\bigoplus_{(M, \tau)} \mathcal{A}(G_n)_{(M, \tau)}.
	\]
	Here, $(M, \tau)$ runs through all equivalence classes of cuspidal data.
	\end{thm}
	
	Let $P$ be a standard parabolic subgroup of $G_n$ with standard Levi subgroup $M$ and $\pi$ an irreducible cuspidal automorphic representation of $M(\bbA_\bbQ)$.
	Put
	\[
	\calA_\cusp(P\bs G_n)_\pi = \{\varphi \in \calA(P \bs G_n) \mid \text{$\varphi_k \in \calA_\cusp(M)_\pi$ for any $k \in K_n$}\}.
	\]
	Here, $\calA_\cusp(M)_\pi$ is the $\pi$-isotypic component of $\calA_\cusp(M)$.
	For an automorphic form $\varphi$, there exists a finite correction of cuspidal data $(M,\tau)$ such that
	\[
	\varphi \in \bigoplus_{(M,\tau)} \calA(G_n)_{(M,\tau)}
	\]
	by Theorem \ref{cusp_supp_decomp}.
	Let $\varphi_P^\cusp$ be the cuspidal part of $\varphi_P$.
	Then, there exists a finite number of irreducible cuspidal automorphic representations $\pi_1,\ldots,\pi_\ell$ of $M_P(\bbA_\bbQ)$ such that
	\[
	\varphi_P^\cusp \in \bigoplus_{j=1}^\ell \bbC[\fraka_P] \otimes \calA_\cusp(P \bs G_n)_{\pi_j}.
	\]
	We say that a set $\cup_M \{\chi_{\pi_1},\ldots,\chi_{\pi_\ell}\}$ is the set of cuspidal exponents of $\varphi$.
	Here, $\chi_{\pi_j}$ is the central character of $\pi_j$.
	For a character $\chi$ of the center of $M_P(\bbA_\bbQ)$, we call the restriction of $\chi$ to $A_P^\infty$ the real part of $\chi$.
	
	Let us now introduce the notion for some induced representations on $G_n(\bbA_\bbQ)$ and $\Sp_{2n}(F_v)$.
	For a character $\mu$ of $\GL_n(\bbA_F)$, we mean an automorphic character, i.e., $\GL_n(F)$ is contained in the kernel of $\mu$.
	Let $\mu$ be a character of $\GL_i(\bbA_F)$ and an irreducible cuspidal automorphic representation $\pi$ of $G_{n-i}(\bbA_\bbQ)$.
	We define the space $\Ind_{P_{i,n}(\bbA_\bbQ)}^{G_n(\bbA_\bbQ)}(\mu|\cdot|^s \boxtimes \pi)$ by the space of smooth functions $\varphi$ on $N_{P_{i,n}}(\bbA_\bbQ)P_{i,n}(\bbQ) \bs G_n(\bbA_\bbQ)$ such that
	\begin{itemize}
	\item $\varphi$ is an automorphic form.
	\item For any $k \in K_n$, the function $\varphi_k$ lies in the $\mu|\cdot|^s\boxtimes\pi$-isotypic component of $L^2_{\mathrm{disc}}(M_{P_{i,n}}(\bbA_\bbQ))$.
	\end{itemize}
	We write 
	\[
	I_{i,n}(s,\mu,\pi) = \Ind_{P_{i,n}(\bbA_\bbQ)}^{G_n(\bbA_\bbQ)} \left(\mu |\cdot|^s \boxtimes \pi\right) \quad \mathrm{and}\quad I_n(s,\mu) = \Ind_{P_{n}(\bbA_\bbQ)}^{G_n(\bbA_\bbQ)} \mu |\cdot|^s.
	\]
	For a place $v$ of $F$, we similarly write
	\[
	I_{i,n,v}(s,\mu_v,\pi_v) = \Ind_{P_{i,n}(F_v)}^{G_n(F_v)} \left(\mu_v |\cdot|^s \boxtimes \pi_v\right) \quad \mathrm{and}\quad I_{n,v}(s,\mu_v) = \Ind_{P_{n}(F_v)}^{G_n(F_v)} \mu_v |\cdot|^s.
	\]
	Here, $\mu_v$ is a character of $\GL_i(F_v)$ and $\pi_v$ is an irreducible representation of $\Sp_{2(n-i)}(F_v)$.
	
	\subsection{Nearly holomorphic automorphic forms}\label{NHAF}

	For an automorphic form $\varphi$ on $G_n(\bbA_\bbQ)$, we say that $\varphi$ is nearly holomorphic if $\varphi$ is $\frakp_{n,-}$-finite.
	The symbol $\calN(G_n)$ denotes the space of nearly holomorphic automorphic forms on $G_n(\bbA_\bbQ)$.
	Put $\calN(G_n)_{(M, \tau)}=\calN(G_n)\cap \mathcal{A}(G_n)_{(M, \tau)}$.
	We say that an irreducible cuspidal automorphic representation $\pi = \bigotimes_v \pi_v$ of $G_n(\bbA_\bbQ)$ is holomorphic if $\pi_v$ is an irreducible unitary highest weight representation of $\Sp_{2n}(F_v)$ for any $v \in \bfa$.
	In \cite[Theorem 1.2]{Horinaga_2}, we determine the cuspidal components of nearly holomorphic automorphic forms as follows:
	
	\begin{prop}\label{decomp_NHAF_cusp_supp}
	Let $P$ be a standard parabolic subgroup of $G_n$ with the standard Levi subgroup $M$.
	\begin{enumerate}
	\item With the above notation, the space $\calN(G_n)_{(M,\pi)}$ is non-zero only if $P$ is associated to $Q_{i,n}$ for some $i$.
	\item Let $\Pi =\mu_1 \boxtimes \cdots \boxtimes\mu_i \boxtimes \pi$ be an irreducible cuspidal automorphic representation of $M_{Q_{i,n}}(\bbA_\bbQ) = (\Res_{F/\bbQ}\GL_1)(\bbA_F)^i \times G_{n-i}(\bbA_\bbQ)$.
	If the space $\calN(G_n)_{(Q_{i,n},\Pi)}$ is non-zero, we have
	\begin{itemize}
	\item $\mu_1 = \cdots = \mu_{i}$.
	\item $\pi$ is a holomorphic cuspidal automorphic representation of $G_{n-i}(\bbA_\bbQ)$.
	\end{itemize}
	\end{enumerate}
	\end{prop}
	
	Let $\mu$ be a character of $\GL_1(\bbA_F)$.
	For simplicity the notation, we denote by $\mu$ the character $\mu \boxtimes \cdots \boxtimes \mu$ of $\GL_1(\bbA_F)^i$.
	In \cite{Horinaga_2}, we determine the structure of the space $\calN(G_n)_{(M,\tau)}$ explicitly under several assumptions.

	\subsection{Modular forms and automorphic forms}
	
	We recall the correspondence of modular forms on the Siegel upper half space and automorphic forms on $G_n(\bbA_\bbQ)$.
	Fix a weight $\rho$ and a congruence subgroup $\Gamma$.
	We embed $\Gamma$ into $G_n(\bbA_{\bbQ,\fini})$ diagonally.
	Let $K_\Gamma$ be the closure of $\Gamma$ in $G_n(\bbA_{\bbQ,\fini})$.
	Then, $K_\Gamma$ is an open compact subgroup of $G_n(\bbA_{\bbQ,\fini})$.
	
	By the strong approximation, one has $G_n(\bbA_\bbQ) = G_n(\bbQ)G_n(\bbR)K_\Gamma$.
	For $f \in N_\rho(\Gamma)$ and $v^* \in \rho^*$, the dual of $\rho$, put
	\[
	\varphi_{f,v^*}(\gamma g_\infty k) = \langle(f|_\rho g_\infty)(\bfi), v^*\rangle, \qquad \gamma g_\infty k \in G_n(\bbQ)G_n(\bbR)K_\Gamma = G_n(\bbA_\bbQ).
	\]
	This is well-defined.
	The map $f \otimes v^* \longmapsto \varphi_{f,v^*}$ induces the inclusion
	\begin{align}\label{corresp_MF_AF_1}
	N_\rho(\Gamma) \otimes \rho^* \longrightarrow \calN(G_n).
	\end{align}
	Put
	\[
	\calN(G_n)_{\rho}^{K_\Gamma} = 
	\left\{
	\varphi \in \calN(G_n) \,\middle|\, 
	\begin{matrix}
	\text{$\varphi$ generates $\rho$ under the action of $K_{n,\infty}$ and}\\
	\text{$\varphi(gk) = \varphi(g)$ for any $g \in G_n(\bbA_\bbQ)$ and $k \in K_\Gamma$}
	\end{matrix}
	\right\}.
	\]
	By the choice of embedding $\mathrm{U}(n) \xhookrightarrow{\quad} \GL_n(\bbC)$, the map (\ref{corresp_MF_AF_1}) induces the isomorphism
	\begin{align}\label{corresp_MF_AF}
	N_\rho(\Gamma) \otimes \rho^* \xrightarrow{\,\,\sim\,\,} \calN(G_n)_\rho^{K_\Gamma}.
	\end{align}
	For details, see \cite[\S 3.2]{HPSS}.
	For a representation generated by $f \in N_\rho(\Gamma)$, we mean the representation generated by $\varphi_{f,v^*}$ with $0 \neq v^*\in \rho^*$.
	Note that the representation is independent of the choice of $v^* \neq 0$.
	
	\section{Computations of unitary highest weight modules with a regular integral infinitesimal character}
	
	In this section, we introduce the parabolic BGG category $\calO^\frakp$ and unitarizable modules in this category.
	For later use, we compute extensions of certain modules and multiplicities of $K_{n,\infty}$-types.
	
	\subsection{parabolic BGG category}
	
	For simplicity the notation, throughout this section, we assume $F = \bbQ$.
	Let $\frakn$ be a nilpotent subalgebra of $\frakg_n$.
	For a $\frakg_n$-module $M$, we say that $M$ is locally $\frakn$-finite if $\calU(\frakn) \cdot v$ is finite-dimensional for any $v \in M$.
	
	We consider the parabolic subalgebra $\frakp = \frakk_n \oplus \frakp_{n,-}$.
	We define the full subcategory $\calO^\frakp$ of the category of $\frakg_n$-modules whose objects $M$  satisfy the following three conditions:
	\begin{itemize}
	\item $M$ is finitely generated.
	\item $M$ decomposes as a direct sum of irreducible finite-dimensional representations of $\frakk_n$.
	\item $M$ is locally $\frakp_{n,-}$-finite.
	\end{itemize}
	The category $\calO^\frakp$ is called the parabolic BGG category $\calO^\frakp$ with respect to $\frakp$.
	For further properties of the BGG category $\calO$ and a parabolic BGG category $\calO^\frakp$, see \cite{cat_o}.
	
	Let us introduce the Verma modules.
	For a $\frakk_n$-dominant weight $\lambda$, let $V_\lambda$ be a model of $\rho_\lambda$.
	We regard $V_\lambda$ as a $\frakp$-module by letting $\frakp_{n,-}$ act trivially.
	Put
	\[
	N(\lambda) = \calU(\frakg_n) \otimes_{\calU(\frakp)} V_\lambda.
	\]
	Then, $N(\lambda)$ has the canonical left $\frakg_n$-module structure.
	The module $N(\lambda)$ is called the (parabolic) Verma module of weight $\lambda$.
	Since $N(\lambda)$ is generated by a highest weight vector, $N(\lambda)$ has the unique irreducible quotient $L(\lambda)$.
	Note that $N(\lambda)$ and $L(\lambda)$ are objects in $\calO^\frakp$. 
	
	For a $\frakg_n$-module $M$, we say that $M$ is a highest weight module if there exists a highest weight vector $v \in M$ such that $v$ generates $M$.
	By definition, Verma modules are highest weight modules.
	Moreover, $N(\lambda)$ has the following universality:
	For a highest weight module $M$ with the highest weight $\lambda$, there exists a surjective homomorphism $N(\lambda) \twoheadrightarrow M$.
	
	For a weight $\lambda$, let $\chi_\lambda$ be the infinitesimal character with the Harish-Chandra parameter $\lambda+\rho$.
	Then, the Verma module $N(\lambda)$ has the infinitesimal character $\chi_\lambda$.
	Note that for $\chi_0$, we mean the infinitesimal character of the trivial representation.
	The infinitesimal characters $\chi_\lambda$ and $\chi_\mu$ are the same if and only if there exists $w \in W_n$ such that $\lambda = w \cdot \mu$.
	Here $\cdot$ is the dot action defined by $w \cdot \mu = w(\mu + \rho) - \rho$.
	For a weight $\lambda$, put $\calO_\lambda = \{w \cdot \lambda\mid w \in W_n\}$.
	We say that $\lambda$ is (dot-)regular if $\#\calO_\lambda = \# W_n$.
	If $\lambda$ is not of (dot-)regular, we say that $\lambda$ is (dot-)singular.
	
	For a nearly holomorphic automorphic form $\varphi$, we consider the $\frakg_n$-module $M$ generated  by $\varphi$ under the right translation.
	Then, $M$ is a $(\frakg_n,K_{n,\infty})$-module.
	By the definition of the parabolic BGG category $\calO^\frakp$, the $\frakg_n$-module $M$ is an object in $\calO^\frakp$.
	
	\subsection{First reduction point and unitarizability}
	
	We recall the definition of the first reduction point in the sense of \cite{EHW_83}.
	Let $\lambda=(\lambda_1,\ldots,\lambda_n)$ be a $\frakk_n$-dominant weight with $\lambda_n=n$.
	We say that a real number $r_0 = r_0(\lambda)$ is the first reduction point if the module $N(\lambda+r_0(-1,\ldots,-1))$ is reducible and $N(\lambda+r(-1,\ldots,-1))$ is irreducible for $r<r_0$.
	Set $p(\lambda)=\#\{i\mid\lambda_i=\lambda_n\}$ and $q(\lambda)  = \#\{i\mid \lambda_i = \lambda_n+1\}$.
	One can compute the first reduction point explicitly by the result of Enright-Howe-Wallach \cite[Theorem 2.10]{EHW_83}.
	
	\begin{thm}\label{first_red_pt}
	Let $\lambda=(\lambda_1,\ldots,\lambda_n)$ be a $\frakk_n$-dominant weight with $\lambda_n=n$.
	Then, the first reduction point $r_0$ equals to $(p(\lambda)+q(\lambda)+1)/2$.
	\end{thm}
	
	Let $r_0$ be the first reduction point.
	Then for $r < r_0$, the irreducible representation $L(\lambda+r(-1,\ldots,-1))$ is unitarizable.
	More precisely, we have the following by \cite[Theorem 2.8]{EHW_83}:
	
	\begin{thm}\label{unitary}
	With the same notation as in Theorem \ref{first_red_pt}, $L(\lambda + r(-1,\ldots,-1))$ is unitarizable if and only if either of the following conditions holds:
	\begin{itemize}
	\item $r \leq (p(\lambda)+q(\lambda)+1)/2$.
	\item $\lambda \in (1/2)\bbZ^n$ and $r \leq p(\lambda)+q(\lambda)/2$.
	\end{itemize}
	\end{thm}
	
	\subsection{Dot-orbits of regular integral weights and unitary highest weight modules}
	
	Let $\lambda = (\lambda_1,\ldots,\lambda_n) \in \bbZ^n$ be a $\frakk_n$-dominant integral weight.
	Let $|\lambda|$ be a multiset $\{|\lambda_1-1|,|\lambda_2-2|,\ldots,|\lambda_n-n|\}$.
	Then, the multiset is invariant under the dot-action, i.e., $|\lambda| = |w \cdot \lambda|$ for any $w \in W_{n}$.
	We then say that $\lambda$ is anti-dominant if $\lambda_n \geq n$.
	We compute the dot-orbits of regular anti-dominant integral weights.
	Note that for any regular integral weight $\lambda$, there exists $\sigma \in W$ such that $\sigma \cdot \lambda$ is anti-dominant.
	Moreover, such an anti-dominant weight is unique in the dot-orbit $\calO_\lambda$.

	\begin{lem}\label{dot_orbit_lem}
	Let $\lambda = (\lambda_1,\ldots,\lambda_n)$ be a regular anti-dominant integral weight and $\sigma$ an element of the Weyl group $W_n$.
	Suppose that the weight $\sigma \cdot \lambda$ is $\frakk_n$-dominant and $L(\sigma\cdot\lambda)$ is unitarizable.
	If $\sigma\cdot \lambda \neq \lambda$, one has $\lambda_n = n+1$.
	\end{lem}
	\begin{proof}
	Put $\omega = \sigma \cdot \lambda = (\omega_1,\ldots,\omega_n)$.
	Suppose that $\omega \neq \lambda$ and $L(\omega)$ is unitarizable.
	By $\omega \neq \lambda$ and the uniqueness of anti-dominant weights in $\calO_\lambda$, one has $\omega_n < n$.
	Set $p = p(\omega)$ and $q = q(\omega)$.
	Since $L(\omega)$ is unitarizable, one has
	\begin{align}\label{ineq}
	n-p-q/2 \leq \omega_n < n.
	\end{align}
	If $\omega_n > n-p$, there exists $n-p+1 \leq j \leq n$ such that $\omega_j - j = 0$.
	Then, $\omega$ is singular.
	This is contradiction.
	Similarly, if $\omega_n < n-p$, one has $q > 0$ by (\ref{ineq}).
	By (\ref{ineq}) and the unitarizability of $L(\omega)$, either of the following statements holds:
	\begin{itemize}
	\item There exists $j$ such that $\omega_j = j$.
	\item There exists $i < j$ such that $\omega_i - i = j - \omega_j$.
	\end{itemize}
	Thus, $\omega$ is singular.
	This is contradiction.
	Hence, one has $\omega_n = n-p$ and in particular $1 \in |\omega| = |\lambda|$.
	Indeed, $|\omega| \ni |\omega_{n-p+1}-(n-p+1)| = |n-p-(n-p+1)| = 1$.
	Since $\lambda$ is anti-dominant, we obtain $\lambda_n = n+1$.
	This completes the proof.
	\end{proof}
	
	For a $\frakk_n$-dominant integral weight $\lambda$, we put 
	\[
	\calO_\lambda^{\mathrm{unit}} = \{\mu \in \calO_\lambda \mid \text{$\mu$ is $\frakk_n$-dominant and $L(\mu)$ is unitarizable}\}.
	\]
	By the proof of the above lemma, we obtain the following corollary:
	
	\begin{cor}\label{orbit}
	Let $\lambda$ be a regular anti-dominant integral weight.
	\begin{enumerate}
	\item If $\lambda_n > n+1$, one has $\calO_\lambda^{\mathrm{unit}} = \{\lambda\}$.
	\item If $\lambda_n = n+1$, one has $\calO_\lambda^{\mathrm{unit}} = \{\lambda^{(0)},\ldots,\lambda^{(p(\lambda))}\}$, where
	\[
	\lambda^{(j)} = (\lambda_1,\ldots,\lambda_{n-j}, n-j,\ldots,n-j).
	\]
	\end{enumerate}
	\end{cor}
	\begin{proof}
	If $\#\calO_\lambda^\unit > 1$, one has $\lambda_n = n+1$ by Lemma \ref{dot_orbit_lem}.
	We may assume $\lambda_{n} = n+1$.
	In this case, the representation $L(\lambda^{(\ell)})$ is unitary for any $1 \leq \ell \leq p(\lambda)$.
	Thus, $\{\lambda^{(0)}, \ldots,\lambda^{p(\lambda)}\} \subset \calO_\lambda^\unit$.
	We prove the converse.
	Take $\mu = (\mu_1,\ldots,\mu_n) \in \calO_\lambda^\unit$.
	By the proof of Lemma \ref{dot_orbit_lem}, we obtain $\mu_n = n - p(\mu)$.
	Since $\lambda$ is regular, the multiset $|\lambda|$ is a set.
	Note that $\{|\lambda_1-1|,\ldots,|\lambda_n-n|\} = \{|\mu_1-1|,\ldots,|\mu_n-n|\}$.
	By the $\frakk_n$-dominance of $\lambda$ and $\mu$, we obtain $\lambda_i - i = \mu_i - i$ for $1 \leq i \leq n-p(\mu)$.
	Thus, one has $\lambda^{(p(\mu))} = \mu$.
	This completes the proof.
	\end{proof}

	\subsection{Multiplicities of certain $K$-types}
	
	In this subsection,  we distinguish $L(\lambda)$ in terms of $K_{n,\infty}$-types in the orbit $\calO_\lambda^\unit$.
	For this, we first recall the embeddings of highest weight modules into principal series representations.
	\begin{thm}[\cite{1989_Yamashita}]\label{emb_ht_wt}
	For a principal series representation
	\[
	\Ind_{B_n(\bbR)}^{G_n(\bbR)} \left( \mu_1|\cdot|^{s_1} \boxtimes \cdots \boxtimes \mu_n|\cdot|^{s_n}\right)
	\]
	with unitary characters $\mu_i$ of $\bbR^\times$,
	the induced representation contains a highest weight representation of weight $(\lambda_1,\ldots\lambda_n)$ if and only if we have
	\[
	s_i = \lambda_{n-i+1}-n+i-1, \mu_i = \mathrm{sgn}^{\lambda_{n-i+1}}
	\]
	for any $1 \leq i \leq n$.
	\end{thm}
	
	For $0 \leq j \leq n$, let $\wedge^j$ be the $j$-th exterior product of the standard representation of $\frakk_n$.
	This is an irreducible representation of $\frakk_n$ with highest weight $(\overbrace{1,\ldots,1}^j,0,\ldots,0)$.
	Put
	\[
	j(\lambda) = \#\{\ell \mid \lambda_1 \equiv \lambda_\ell \mod 2\}.
	\]
	The following statement follows from the Littlewood-Richardson rule.
	
	\begin{lem}\label{occur}
	For a $\frakk_n$-dominant integral weight $\lambda$, one has
	\[
	\mathrm{Hom}_{\frakk_n}(\wedge^{j(\lambda)}\otimes\mathrm{det}^{\lambda_1-1}, N(\lambda)|_{\frakk_n}) \neq 0.
	\]
	\end{lem}
	\begin{proof}
	For an integral weight $\omega = (\omega_1,\ldots,\omega_n)$, we consider the following two step operation:
	\begin{itemize}
	\item[Step 1.] Put $\omega'_1 = \omega_1$.
	For $2 \leq i \leq n$, set
	\[
	\omega_i' =
	\begin{cases}
	\omega_{i-1} &\text{if $\omega_{i-1} - \omega_i$ is even}\\
	\omega_{i-1} -1 & \text{if $\omega_{i-1} - \omega_i$ is odd}.
	\end{cases}
	\]
	\item[Step 2.]
	Consider the set $X = X(\omega) = \{i \mid 2 \leq i \leq n, \omega_{i-1} \neq \omega_i'\}$.
	Let $a$ be the maximal element in $X$.
	We define a new set $X' = X'(\omega)$ by $X' = X$ if $\# X$ is even and by $X' = X \setminus \{a\}$ if $\# X$ is odd.
	Put
	\[
	\begin{cases}
	\omega_{i}'' = \omega_{i}' & \text{if $i \not\in X'$}\\
	\omega_{i}'' = \omega_{i}'+1 & \text{if $i \in X'$}.
	\end{cases}
	\]
	\end{itemize}
	We define a map $g \colon \bbZ^n \longrightarrow \bbZ^n$ by $g((\omega_1,\ldots,\omega_n)) = (\omega_1'',\ldots,\omega_n'')$.
	Note that the image of $\frakk_n$-dominant weight is $\frakk_n$-dominant.
	We denote by $g^\ell$ the $\ell$-th composite of $g$.
	Set $g^\ell(\lambda) = (\lambda_{1,\ell},\ldots,\lambda_{n,\ell})$ and $a_\ell = \sum_{i=1}^n (\lambda_{i,\ell}-\lambda_{i,\ell+1})$.
	Then, by definition, $a_i \in 2 \bbZ$.
	By the well-known correspondence of young diagrams and irreducible finite-dimensional representations of $\frakk_n$, one can show that $(a_1,\ldots,a_n)$ is $\frakk_n$-dominant.
	By the definition of $g$ and the Littlewood-Richardson rule, the irreducible representation of $\frakk_n$ with highest weight $g^{n-1}(\lambda)$ occurs in the tensor product representation $\rho_\lambda \otimes \rho_{(a_1,\ldots, a_n)}$ of $\frakk_n$.
	
	We next compute the weight $g^{n-1}(\lambda) = (\lambda_{1,n-1},\ldots,\lambda_{n,n-1})$.
	By the construction, $g^{n-1}(\lambda)$ is of the form $(\lambda_1,\ldots,\lambda_1,\lambda_1-1,\ldots,\lambda_1-1)$.
	Indeed, by induction on $\ell$, one has $\lambda_{1}-\lambda_{1+\ell,\ell} \leq 1$ for any $1 \leq \ell$.
	Thus, $\lambda_{1}-\lambda_{n,n-1} \leq 1$.
	We claim $j(\omega) = j(g(\omega))$ for any $\frakk_n$-dominant weight $\omega = (\omega_1,\ldots,\omega_n)$.
	Set $g(\omega) = (\omega_1'',\ldots,\omega_n'')$.
	We write $X'(\omega) = \{x_1,\ldots,x_{2m}\}$ with $x_1 < x_2 < \cdots < x_{2m}$.
	Then, for any $x_i \leq \ell \leq x_{i+1}$ with $0 \leq i \leq 2m-1$, one has $\omega_1 \equiv \omega_\ell +(1+(-1)^{i+1})/2$ mod $2$.
	Here, $x_0 = 1$.
	In particular, for any $1 \leq \ell \leq m$, we have $\omega_{x_{2\ell - 1}} \not \equiv \omega_{x_{2\ell}}$ mod $2$.
	By $\omega_x \equiv \omega_x'+1$ mod $2$ for any $x \in X'(\omega)$, we have $j(\omega) = j(g(\omega))$.
	Hence we obtain
	\[
	g^{n-1}(\lambda) = (\overbrace{\lambda_1,\ldots,\lambda_1}^{j(\lambda)},\lambda_1-1,\ldots,\lambda_1-1).
	\]

	By the claim, we see that $\rho_{g^{n-1}(\lambda)} = \wedge^{j(\lambda)} \otimes \det^{\lambda_1-1}$ occurs in $\rho_\lambda \otimes \rho_{(a_1,\ldots,a_n)}$.
	The space $\calU(\frakp_{n,+})$ decomposes as
	\[
	\calU(\frakp_{n,+}) = \bigoplus_{(b_1,\ldots,b_n) \in (2\bbZ)^n, b_1 \geq \cdots \geq b_n} \rho_{(b_1,\ldots,b_n)}
	\]
	as a representation of $\frakk_n$.
	Thus, $\rho_{(a_1,\ldots,a_n)}$ occurs in $\calU(\frakp_{n,+})$.
	Note that the restriction of $N(\lambda)$ to $\frakk_n$ is semisimple and $N(\lambda)|_{\frakk_n} = \calU(\frakp_{n,+})|_{\frakk_n} \otimes_\bbC \rho_\lambda$.
	We then have
	\[
	\mathrm{Hom}_{\frakk_n}(\wedge^{j(\lambda)}\otimes\mathrm{det}^{\lambda_1-1}, N(\lambda)|_{\frakk_n}) \neq 0.
	\]
	This completes the proof.
	\end{proof}
	
	Since a weight of $\wedge^\ell$ is a permutation of $(\overbrace{1,\ldots,1}^j,0,\ldots,0)$, one has the following:
	
	\begin{prop}\label{vanish_K_type}
	For a regular anti-dominant integral weight $\lambda$ and $1 \leq j \leq p(\lambda)$, one has
	\[
	\dim_\bbC \mathrm{Hom}_{\frakk_n}(\wedge^{j(\lambda)} \otimes \mathrm{det}^{\lambda_1-1}, L(\lambda)|_{\frakk_n}) = 1
	\]
	and
	\[
	\mathrm{Hom}_{\frakk_n}(\wedge^{j(\lambda)}\otimes\mathrm{det}^{\lambda_1-1}, L(\lambda^{(j)})|_{\frakk_n}) = 0.
	\]
	\end{prop}
	\begin{proof}
	Let $I_n(\mu_1,\ldots,\mu_n)$ be the principal series representation
	\[
	\Ind_{B_n(\bbR)}^{G_{n}(\bbR)} (\mu_1\boxtimes \cdots \boxtimes \mu_n).
	\]
	Here, $\mu_i$ are real valued characters of $\bbR^\times$.
	Take $\varepsilon_j \in \{0,1\}$ such that $\mu_j(-1)=(-1)^{\varepsilon_j}$. 
	Through the weight structure of $\wedge^j$, one can find that the Hom space
	\[
	\mathrm{Hom}_{\frakk_n}(\wedge^j, I_n(\mu_1,\ldots,\mu_n)|_{\frakk_n})
	\]
	is non-zero if and only if $\sum_{\ell=1}^n\varepsilon_\ell=j$ by the Frobenius reciprocity.
	By the Frobenius reciprocity, one has the multiplicity free, i.e., $\dim_\bbC \mathrm{Hom}_{\frakk_n}(\wedge^j, I_n(\mu_1,\ldots,\mu_n)|_{\frakk_n}) \leq 1$.
	
	For any $\omega \in \calO^{\mathrm{unit}}_\lambda$, the highest weight module $L(\omega)$ occurs in constituents of the induced representation $I_n(\mathrm{sgn}^{\lambda_n}|\cdot|^{\lambda_n-n}, \ldots,\mathrm{sgn}^{\lambda_1}|\cdot|^{\lambda_1-1})$ by Theorem \ref{emb_ht_wt}.
	Then, the statement follows from Lemma \ref{occur} and the above multiplicity free.
	This completes the proof.
	\end{proof}
	
	For a $\frakk_n$-type $\sigma$, put
	\[
	\calO_{\lambda}^{\mathrm{unit}}(\sigma) = \{\pi \in \calO_{\lambda}^{\mathrm{unit}} \mid \mathrm{Hom}_{\frakk_n}(\sigma, \pi|_{\frakk_n}) \neq 0\}.
	\]
	\begin{cor}\label{distinguish_mod}
	For a regular anti-dominant integral weight $\lambda$, one has
	\[
	\calO_{\lambda}^{\mathrm{unit}}(\mathrm{det}^{\lambda_1-1}\otimes\wedge^{j(\lambda)}) = \{L(\lambda)\}.
	\]
	\end{cor}
	\begin{proof}
	The statement follows immediately from Corollary \ref{orbit} and Proposition \ref{vanish_K_type}.
	\end{proof}

	\subsection{Extensions of certain modules}
	
	Fix an odd integer $i$.
	Let $\lambda = (\lambda_1,\ldots,\lambda_n)$ be a $\frakk_n$-dominant integral weight such that $\lambda_{n-i+1} = \cdots = \lambda_{n} = n-(i-3)/2$.
	Put $\lambda' = (\lambda_1',\ldots,\lambda_n') = (\lambda_1,\ldots,\lambda_{n-i}, n-(i+1)/2,\ldots,n-(i+1)/2)$.
	By $|\lambda| = |\lambda'|$ as the multisets, the weights $\lambda$ and $\lambda'$ have the same dot-orbit.
	
	\begin{lem}\label{Ext}
	One has
	\[
	\dim_\bbC \mathrm{Ext}_{\calO^\frakp}(L(\lambda'), L(\lambda)) =1.
	\]
	Moreover an indecomposable module $M$ with a non-trivial exact sequence
	\[
	0 \longrightarrow L(\lambda) \longrightarrow M \longrightarrow L(\lambda') \longrightarrow 0
	\]
	is isomorphic to $N(\lambda')$.
	\end{lem}
	\begin{proof}
	Set $\lambda'' = \lambda' + ((i+1)/2,\ldots,(i+1)/2)$.
	Then, $\lambda''$ satisfies the condition as in Theorem \ref{first_red_pt}, i.e., the $n$-th entry of $\lambda''$ is $n$.
	By $p(\lambda'') = i$ and $q(\lambda'') = 0$, the weight $\lambda'$ corresponds to the first reduction point $r_0(\lambda'')$.
	Thus, the Verma module $N(\lambda')$ is reducible and $N(\lambda)$ is irreducible.
	Let $\omega = (\omega_1,\ldots,\omega_n)$ be a $\frakk_n$-dominant integral weight such that $L(\omega)$ is a constituent of $N(\lambda')$.
	Then, there exists $w \in W_{n}$ such that $\omega = w \cdot \lambda'$ and $\omega \leq \lambda'$.
	We then have $\lambda'_j \leq \omega_j$ for any $j$.
	The multiplicity of $|\lambda_j-j|$ in the multiset $|\lambda|$ is one if and only if $1 \leq j \leq n-i+2$ or $j = n - (i-3)/2$.
	Thus, $\omega$ satisfies the following conditions:
	\begin{itemize}
	\item $\omega_j = \lambda_j$ for $1 \leq j \leq n-i$.
	\item For $j \geq n-i+3$, there exist $k$ and $\ell$ such that $\omega_k-k=\ell-\omega_\ell = |\lambda_j - j|$.
	\item For $j = n-i+1, n-i+2$, there exists $k$ such that $|\omega_k-k| = |\lambda_j-j|$.
	\end{itemize}
	Indeed, it suffices to check the first condition.
	Suppose that there exist $j \leq n-i$ and $k$ such that $\omega_k - k = - (\lambda_j - j)$.
	Then, $\omega_k = k + j - \lambda_j \leq k + n-i - (n - (i-3)/2) = k - (i+3)/2 < n - (i+1)/2 \leq \lambda_n'$.
	This contradicts to $\lambda_k' \leq \omega_k$.
	We note that $\lambda$ is a minimal element in $\calO_\lambda^\unit$.
	The candidates of $\omega$ are
	\[
	\lambda, \qquad \lambda', \qquad (\lambda_1,\ldots,\lambda_{n-i}, n-(i-1)/2, n-(i+1)/2, \ldots, n-(i+1)/2)
	\]
	and
	\[
	(\lambda_1,\ldots,\lambda_{n-i}, n-(i+1)/2,\ldots,n-(i+1)/2,n-(i-1)/2).
	\]
	Since $L(\omega)$ occurs in a constituent of $N(\lambda')$, the representation $\rho_\omega$ occur in the restriction of $N(\lambda')$ to $\frakk_n$.
	However, the last two candidates of $\omega$ do not occur in the restriction $N(\lambda)|_{\frakk_n}$.
	Thus, any constituent of $N(\lambda)$ is of the form $L(\lambda)$ and $L(\lambda')$.
	The irreducible representations $\rho_\lambda$ and $\rho_{\lambda'}$ of $\frakk_n$ have the multiplicity one in $N(\lambda')$.
	Hence the multiplicities of $L(\lambda)$ and $L(\lambda')$ in the constituent of $N(\lambda')$ are at most one.
	Since $N(\lambda')$ is reducible and $L(\lambda')$ is the irreducible quotient of $N(\lambda')$, we obtain a non-split exact sequence
	\[
	0 \longrightarrow L(\lambda) \longrightarrow N(\lambda') \longrightarrow L(\lambda') \longrightarrow 0.
	\]
	By applying $\mathrm{Hom}_{\calO^\frakp}(\,\cdot\,,L(\lambda))$ to this short exact sequence, we obtain the following long exact sequence
	\begin{align*}
	&0 \longrightarrow \mathrm{Hom}_{\calO^\frakp}(L(\lambda'),L(\lambda)) \longrightarrow \mathrm{Hom}_{\calO^\frakp}(N(\lambda'),L(\lambda)) \longrightarrow \mathrm{Hom}_{\calO^\frakp}(L(\lambda),L(\lambda))\\ 
	&\qquad \longrightarrow \mathrm{Ext}^1_{\calO^\frakp}(L(\lambda'),L(\lambda)) \longrightarrow \mathrm{Ext}^1_{\calO^\frakp}(N(\lambda'),L(\lambda)) \longrightarrow \mathrm{Ext}^1_{\calO^\frakp}(L(\lambda),L(\lambda)) \longrightarrow \cdots.
	\end{align*}
	By definition, we have $\mathrm{Hom}_{\calO^\frakp}(N(\lambda'),L(\lambda)) = 0$.
	Consider an extension
	\[
	0 \longrightarrow L(\lambda) \longrightarrow M \longrightarrow N(\lambda') \longrightarrow 0.
	\]
	Let $v$ be a weight vector in $M$ of weight $\lambda'$ such that the image of $v$ in $N(\lambda')$ is a non-zero highest weight vector.
	Since the weight of $v$ is highest in $M$, there exists a splitting $N(\lambda') \longrightarrow M$ by the universality of $N(\lambda')$.
	Thus, the short exact sequence splits, i.e., $\mathrm{Ext}^1_{\calO^\frakp}(N(\lambda'),L(\lambda))=0$.
	We then obtain $\mathrm{Ext}^1_{\calO^\frakp}(L(\lambda'),L(\lambda)) \cong \mathrm{Hom}_{\calO^\frakp}(L(\lambda),L(\lambda))$.
	This is of dimension one.
	This completes the proof.
	\end{proof}

\section{Siegel Eisenstein series}
	In this section, we compute the cuspidal components and exponents of Siegel Eisenstein series via the Siegel-Weil formula.
	We also show the near holomorphy of certain Siegel Eisenstein series.
	
	\subsection{Siegel-Weil formula}
	
	Let $m$ be a positive even integer.
	Set $s_0 = (m-n-1)/2$.
	Let $V$ be a $m$-dimensional quadratic space over $F$ and $\calS(V(\bbA_F)^n)$ the space of Schwartz functions on $V(\bbA_F)^n$.
	We denote by $\omega(V) = \omega_\psi(V) = \bigotimes_{v}\omega_{\psi, v}(V_v)$ the Weil representation of $G_{n}(\bbA_\bbQ) \times \mathrm{O}(V)(\bbA_F)$ on $\calS(V(\bbA_F)^n)$.
	Here $V_v$ is the $v$-completion of $V$ for a place $v$ of $F$.
	For $\varphi \in \calS(V(\bbA_F)^n)$, set
	\[
	\theta(g,h;\varphi) = \sum_{v \in V(F)^n} \omega(g)\varphi(v \cdot h), \qquad g \in G_n(\bbA_\bbQ), h \in \mathrm{O}(V)(\bbA_F).
	\]
	The function $\theta$ is called the theta function.
	Put
	\[
	I(g,\varphi) = \int_{\mathrm{O}(V)(F) \bs \mathrm{O}(V)(\bbA_F)} \theta(g,h;\varphi) \, dh, \qquad g \in G_n(\bbA_\bbQ).
	\]
	The following condition (W) is called the Weil's convergence condition:
	\begin{align}\label{W}
	\tag{W}
	\begin{cases}
	\text{$V$ is anisotropic} \\
	\text{$m-r > n+1$},
	\end{cases}
	\end{align}
	where $r$ is the Witt index of $V$.
	If $V$ satisfies the condition \eqref{W}, the theta integral $I(\,\cdot\,,\varphi)$ converges absolutely.

	For $\varphi \in \calS(V(\bbA_F)^n)$, set
	\[
	f_\varphi (g) = \omega_{\psi}(g)\varphi(0).
	\]
	Let $\chi_V$ be the quadratic character associated to $V$.
	Then $f_\varphi$ is an element of $I_n(s_0,\chi_V)$.
	We denote by $f_{s,\varphi}$ the standard section of $I_n(s,\chi_V)$ such that
	\[
	f_{s_0,\varphi} = f_\varphi.
	\]
	For a standard section $f_s$ of $I_n(s,\mu)$, put
	\[
	E(g,s,f) = \sum_{\gamma \in P_n(\bbQ) \bs G_n(\bbQ)} f_{s}(\gamma g), \qquad E(g,s,\varphi) = E(g,s,f_\varphi).
	\]
	The Siegel-Weil formula states the relationship between $E(g,s_0,\varphi)$ and $I(g,\varphi)$.
	In this paper, we use the following Siegel-Weil formula due to Kudla-Rallis \cite{1988_Kudla-Rallis}.
	
	\begin{thm}\label{SW_formula}
	Suppose that $V$ satisfies the condition \eqref{W}.
	One has
	\[
	E(s_0,\varphi) = cI(g,\varphi)
	\]
	and
	\[
	c = 
	\begin{cases}
	1 & \text{If $m > n+1$} \\
	2 & \text{If $m \leq n+1$}.
	\end{cases}
	\]
	\end{thm}
	
	
	\subsection{The representation $R_n$}
	
	Let $V_v$ be a $m$-dimensional quadratic space over $F_v$ with $m \in 2\bbZ_{\geq 0}$.
	The character $\chi_V$ denotes the quadratic character associated to $V$.
	The map $\varphi \longmapsto f_\varphi$ induces a $\Sp_{2n}(F_v)$-intertwining map
	\[
	\omega_{\psi,v}(V_v) \longrightarrow I_n(s_0, \chi).
	\]
	We denote by $R_n(V_v)$ the image of the intertwining map.
	Then, $R_n(V_v)$ can be viewed as the $\mathrm{O}(V_v)(F_v)$-coinvariants of $\omega_{\psi,v}(V_v)$.
	
	\begin{prop}\label{local_ind_rep}
	With the above notation, we obtain the following:
	\begin{enumerate}
	\item For a non-archimedean place $v$, if $\chi^2 = \mathbf{1}$ and $s_0 \geq 0$, one has
	\[
	I_{n,v}(s_0,\chi) = R_n(V_1) + R_n(V_2).
	\]
	Here $V_1$ and $V_2$ are the $m$-dimensional inequivalent quadratic spaces over $F_v$ with $\chi = \chi_{V_1} = \chi_{V_2}$.
	\item For an archimedean place $v$ of $F$, let $V$ be a $m$-dimensional quadratic space over $F_v = \bbR$ with the signature $(m,0)$ or $(m-2,2)$.
	If $s_0 >0$, the representation $R_n(V_v)$ contains $L(m,\ldots,m)$.
	\item For $v \in \bfa$ and $s_0 > 0$, the space of $\frakp_{n,-}$-finite vectors in $I_{n,v}(s_0,\chi)$ forms
	\[
	\begin{dcases}
	0&\text{if $\chi = \mathrm{sgn}^{m+1}$}\\
	L(m,\ldots,m) & \text{if $\chi = \mathrm{sgn}^m$}.
	\end{dcases}
	\]
	\end{enumerate}
	\end{prop}
	\begin{proof}
	The statement (1) and (2) are proved in \cite[Proposition 5.3]{1994_Kudla-Rallis} and \cite[Proposition 2.1]{1990_Kudla-Rallis}, respectively.
	For the last statement, see \cite[Corollary 6.5]{Horinaga_2}.
	\end{proof}
	
	For a quadratic space $V$ over $F$ and a place $v$ of $F$, let $V_v$ be the $v$-completion of $F$.
	Put
	\[
	R_n(V) = \bigotimes_v R_n(V_v)
	\]
	where $v$ runs over all places of $F$.
	
	\subsection{Holomorphy of Siegel Eisenstein series for $s_0 > 1$ or $F \neq \bbQ$}
	
	Let $f_s = \bigotimes_v f_{v,s}$ be a standard section of $I_n(s,\mu)$.
	For a representation $M$ of $\frakg_n$, we denote by $M_{\frakp_{n,-}\fin}$ the space of $\frakp_{n,-}$-finite vectors in $M$.
	Put $s_0 = (m-n-1)/2$ with non-negative even integer $m$.
	
	\begin{lem}\label{s>1}
	For $s_0 > 1$, the map
	\[
	I_n(s_0,\mu)_{\frakp_{n,-}\fin} \longrightarrow \calA(G_n)
	\]
	defined by
	\[
	f_s \longmapsto E(\,\cdot\,,s_0,f)
	\]
	is injective and intertwining under the action of $G_n(\bbA_\bbQ)$.
	If $F \neq \bbQ$ or $\mu^2 \neq \mathbf{1}$, the same statement holds for $s_0 > 0$.
	\end{lem}
	\begin{proof}
	Suppose $\mu^2 \neq \mathbf{1}$.
	This case is clear by the holomorphy of intertwining operators as in \cite{1992_Ikeda}.
	
	Next, we suppose $\mu^2 = \mathbf{1}$ and $s_0 > 1$.
	By Proposition \ref{local_ind_rep} (3), we have
	\begin{align}\label{finite_vectors}
	I_n(s_0, \mu)_{\frakp_{n,-}\fin} = \left(\mathop{\bigotimes}_{v < \infty} I_{n,v}(s_0, \mu_v)\right) \otimes L(m/2,\ldots,m/2).
	\end{align}
	We claim that the representation $I_n(s_0,\mu)_{\frakp_{n,-}\fin}$ is contained in $\sum_V R_n(V)$ where $V$ runs through the $m$-dimensional quadratic spaces over $F$ such that $V$ satisfies the Weil's convergence condition \eqref{W} and $\chi_V = \mu$.
	Take a function $f = \bigotimes_v f_v$ in $I_n(s_0,\mu)_{\frakp_{n,-}\fin}$.
	We may assume that each local functions $f_v$ lie in $R_n(V_v)$ for some quadratic space $V_v$ over $F_v$ by Proposition \ref{local_ind_rep} (1).
	Let $\varepsilon_v$ be the Hasse invariant of $V_v$.
	We denote by $V(a,b)$ the non-degenerate real quadratic space with the signature $(a,b)$.
	The Hasse invariants of $V(m,0)$ and $V(m-2,2)$ are $1$ and $-1$, respectively.
	Fix an archimedean place $w$.
	Then, there exists the quadratic space $W = \bigotimes_v W_v$ over $F$ such that $W_v \cong V_v$ for any non-archimedean place $v$, $W_v \cong V(m,0)$ for any archimedean place $v \neq w$ and
	\[
	W_w = 
	\begin{cases}
	V(m,0) & \text{if $\prod_{v < \infty} \vep_v = 1$}\\ 
	V(m-2,2) & \text{if $\prod_{v < \infty} \vep_v = -1$}.
	\end{cases}
	\]
	By Proposition \ref{local_ind_rep} (2) and $m>n+3$, the quadratic space $W$ satisfies the condition \eqref{W} and $f \in R_n(W)$.
	Hence the claim holds.
	By the claim, the theta integral converges absolutely.
	This states that the theta integral is an intertwining map under the action of $G_n(\bbA_\bbQ)$.
	Hence we obtain the following diagram:
	\[
	\begin{CD}
	\calS(W(\bbA_F)^n) @>>> \calN(G_n)\\
	@VVV @AAA\\
	R_n(W) @>>> I_n(s_0,\chi_W)_{\frakp_{n,-}\fin}.
	\end{CD}
	\]
	Here the upper horizontal line is given by $\varphi \longmapsto I(\,\cdot\,,\varphi)$, the left vertical line is the canonical surjective morphism and the right vertical line is given by $f \longmapsto E(\,\cdot\,,s_0,f)$.
	By the definition of the theta integral, it factors through the $\mathrm{O}(W)(\bbA_F)$-coinvariants $R_n(W)$ of $\omega_{\psi}$.
	By the Siegel-Weil formula Theorem \ref{SW_formula}, the diagram is commutative.
	Hence the right vertical map is intertwining under the action of $G_n(\bbA_\bbQ)$.
	For the injectivity, we consider the constant term of $I(\,\cdot\,,\varphi)$ along $P_{n}$.
	By the straightforward computation, one has $I(\,\cdot\,,\varphi)_{P_{n}}(g) = \omega(g)\varphi(0)$.
	Thus the right vertical line is injective.
	
	For the case $F \neq \bbQ$, it suffices to show that the space of induced representations (\ref{finite_vectors}) are contained in $\sum_{V} R_n(V)$ where $V$ runs over all quadratic spaces over $F$ with dimension $m$ such that $V$ satisfies the condition \eqref{W}.
	The proof is similar.
	Take $f = \bigotimes f_v$ in the induced representation (\ref{finite_vectors}).
	We may assume $f_v \in R_n(V_v)$ for any place $v$.
	Let $\vep_v$ be the Hasse invariant of $V_v$.
	Fix an archimedean place $w$.
	If $\prod_{v < \infty} \vep_v =1$, we can find a positive definite quadratic space $W$ over $F$ such that $f_v \in R_n(W_v)$.
	If $\prod_{v < \infty} \vep_v = -1$, we can find a quadratic space $W$ over $F$ such that $W_v \cong V_v$ for any non-archimedean place $v$, $W_v$ is positive definite for any archimedean place $v \neq w$, $W_w$ is of signature $(m-2,2)$ and $f \in R_n(W)$.
	Then, $W$ is anisotropic.
	We obtain the claim.
	This completes the proof.
	\end{proof}
	
	
%
%

	In the following of this section, we assume $F = \bbQ$.
	
	\subsection{Near holomorphy of Siegel Eisenstein series for $s=1$}
	
	\begin{prop}\label{s=1}
	Let $f_s = \bigotimes_v f_{v,s}$ be a standard section of $I_n(s,\mu)$ such that $f_1 \in I_n(1, \mu)_{\frakp_{n,-}\fin}$.
	Then the Eisenstein series $E(\,\cdot\,,1,f)$ is nearly holomorphic, i.e., there exists $\ell$ such that $\frakp_{n,-}^\ell \cdot E(\,\cdot\,,1,f) = 0$.
	\end{prop}
	\begin{proof}
	We may assume that there exists a $(n+3)/2$-dimensional quadratic space $W_v$ over $F_v$ such that $f_{v,1} \in R(W_v)$ for any place $v$ of $F$.
	Here $W_v$ is positive definite for the archimedean place $v$.
	Let $\varepsilon_v$ be the Hasse invariant of $W_v$.
	If $\prod_{v} \varepsilon_v = 1$, the corresponding Eisenstein series $E(\,\cdot\,,1,f)$ generates the representation $\bigotimes_v R(W_v)$, by the same method as in the proof of Lemma \ref{s>1}.
	Then, the archimedean component is a highest weight representation.
	In particular, $E(\,\cdot\,,1,f)$ is nearly holomorphic.
	
	Suppose $\prod_v\varepsilon_v = -1$.
	Let $V$ be the quadratic space over $F=\bbQ$ with dimension $(n+3)/2$ such that $V_v \cong W_v$ for non-archimedean place $v$.
	Then, for the archimedean place $v$, we may assume that the quadratic space $V_v$ has the signature $(n+1,2)$.
	We consider the constant term of $E(\,\cdot\,,1,f)$ along $P_{1,n}$.
	For $(t,g) \in \GL_1(\bbA_F) \times G_{n-1}(\bbA_\bbQ)$ and $k \in K_n$, one has
	\begin{align}\label{P_1_const_term}
	E((t,g)k,s,f)_{P_{1,n}} 
	&= \mu(t)|t|^{s+(n+1)/2} E\left(g,s+\frac{1}{2},\iota^*r(k)f\right)\\
	\nonumber&\qquad + \mu(t)^{-1}|t|^{-s+(n+1)/2}E \left(g,s-\frac{1}{2},\iota^*U(s,\mu)r(k)f\right)
	\end{align}
	where $\iota$ is the embedding $G_{n-1} \xhookrightarrow{\quad} P_{1,n} \xhookrightarrow{\quad} G_n$ and $U(s,\mu)$ is the intertwining integral defined by
	\[
	U(s,\mu)f_s = \int_{U_1(\bbA_F)} f_s(w_1ug)\, du
	\]
	for
	\[
	w_1 = \begin{pmatrix} 0 & 0 & -1 & 0 \\ 0 & \mathbf{1}_{n-1} & 0 & 0 \\ 1 & 0 & 0 & 0 \\ 0 & 0 & 0 & \mathbf{1}_{n-1} \end{pmatrix}, \qquad U_1 = \left\{u = \begin{pmatrix}1 & x & y & 0 \\ 0 & \mathbf{1}_{n-1} & 0 & 0 \\ 0 & 0 & 1 & 0 \\ 0 & 0 & -{^tx} & \mathbf{1}_{n-1} \end{pmatrix} \,\middle|\, x \in \bbA_F^{n-1}, y \in \bbA_F\right\}.
	\]
	We denote by $U_v(s,\mu)f_{v,s}$ the local intertwining integral so that $\bigotimes_v U_v(s,\mu)f_{v,s} = U(s,\mu)f_s$.
	Note that the local intertwining integral $U_v(s,\mu)$ converges absolutely for $\mathrm{Re}(s) \gg 0$.
	Moreover, it is holomorphic and non-zero for $\mathrm{Re}(s) > 0$.
	See \cite[pp.\,91]{1987_PS-R}.
	Let $\infty$ be the archimedean place of $\bbQ$.
	Let $F_s$ be the standard section of $I_{n-1}(s,\mu)$ such that $F_{s_0'} = \iota^* U(s,\mu) r(k)f_s|_{s=s_0'}$ for some $s_0'$ with $\mathrm{Re}(s_0') \gg 0$.
	We claim that there exists a non-zero constant $c$ such that
	\[
	\mathrm{Res}_{s = 1/2} E(g,s,F) = c E(g,1/2,\iota^*U(s,\mu)r(k)f).
	\]
	For a non-zero standard section $h_s$ of $I_\infty(s,\mu)$ of weight $k$, the integral $U_\infty(s,\mu)h_s(1)$ is a non-zero multiple of
	\begin{align}\label{U_wt_k}
	\frac{\Gamma(s)}{\Gamma((s+k+(n+1)/2)/s)\Gamma((s+(n+1)/2-k)/2)}
	\end{align}
	by \cite[(1.22)]{1994_Kudla-Rallis} and \cite[Lemma 4.6]{1988_Kudla-Rallis}.
	Substitute $k = (n+3)/2$.
	Then, $U_\infty(s,\mu)h_s$ has a simple zero at $s=1/2$.
	Hence the integral $U_\infty(s,\mu)f_{s,\infty}$ has a simple zero at $s=1$.
	Indeed, at $s = s_0$, $f_s$ can be written as a sum of right translations of a non-zero function of weight $(n+3)/2$.
	Put
	\[
	U^*_v(s,\mu)f_s =
	\begin{dcases}
	\iota^*U_v(s,\mu)f_{v,s} & \text{if $v$ is non-archimedean.}\\
	\frac{\Gamma((s+n+4)/2)\Gamma((s-1)/2)}{\Gamma(s)} \iota^* U_v(s,\mu)f_{v,s} & \text{if $v$ is archimedean.}
	 \end{dcases}
	\]
	For an unramified place $v$, by \cite[(1.23)]{1994_Kudla-Rallis}, one has
	\begin{align}\label{unram_U}
	\iota^*U(s,\mu)f^\circ_{v,s}(1) = \frac{L_v(s+(n-1)/2,\mu)L_v(2s,\mu^2)}{L_v(s+(n+1)/2,\mu)L_v(2s+n-1,\mu^2)}f_{v,-s}^\circ(1)
	\end{align}
	where $f_{v,s}^\circ$ is the unramified section of $I_{n,v}(s,\mu)$.
	Thus, the meromorphic section $U^*(s,\mu)f_s$ is holomorphic for $s=1$.
	By Lemma \ref{s>1} and \cite[Theorem 1.1]{1994_Kudla-Rallis}, $E(\,\cdot\,,s - 1/2, \iota^*U(s,\mu)r(k)f)$ has at most simple pole at $s=1$.
	We then have
	\begin{align*}
	\lim_{s \rightarrow 1} E(g,s-1/2,\iota^*U(s,\mu)f) 
	&= \lim_{s \rightarrow 1} \frac{\Gamma(s)(s-1)}{\Gamma((s+n+4)/2)\Gamma((s-1)/2)}(s-1)^{-1}E(g,s-1/2,U^*(s,\mu)f)\\
	&= \frac{c}{2 \cdot \Gamma((n+5)/2)} \mathrm{Res}_{s=1} E(g,s-1/2,F)
	\end{align*}
	with some non-zero constant $c$.
	Hence the claim holds.
	Let $V_0$ be the complementary space of $V$ in the sense of \cite[pp.\,34]{1994_Kudla-Rallis}.
	By \cite[Corollary 6.3]{1994_Kudla-Rallis}, the constant term of $\mathrm{Res}_{s=1} E(g,s-1/2,F)$ along $P_{n-1}$ lies in $R_{n-1}(V_0) \subset I_{n-1}(-1/2,\mu)$.
	Thus, the constant term of $E(g,s,f)$ along the Borel subgroup is an element of weight $k$ in a direct sum of principal series representations.
	Comparing the scalar $K_{n,\infty}$-types of principal series representations and degenerate principal series representations, the constant term lies in
	\[
	I_n(1,\chi_V) \oplus I_n(-1,\chi_V)
	\]
	of weight $(n+3)/2$.
	Note that the $K_{n,\infty}$-type with highest weight $((n+3)/2,\ldots,(n+3)/2)$ occur in $I_{n,\infty}(-1,\mu)_{\frakp_{n,-}\fin}$ by \cite[Lemma 3.5]{Horinaga_2}.
	We also note that $E(\,\cdot\,,s,f)$ concentrates on the Borel subgroup.
	Hence the Eisenstein series $E(\,\cdot\,,1,f)$ is nearly holomorphic.
	This completes the proof.
	\end{proof}
	
	\begin{rem}
	In the above proof, we use the formula of $d_{n,v}(s,\ell)$ as in \cite[Lemma 4.6]{1988_Kudla-Rallis}.
	We should note that there is a typo in this formula.
	The correct one is
	\[
	d_{n,v}(s,\ell) = (\sqrt{-1})^{nk}2^{-ns}(2\pi)^n\pi^{n(n-1)/2}	 \frac{\Gamma_n(s)}{\Gamma_n((s+\rho_n+\ell)/2)\Gamma_n((s+\rho_n-\ell)/2)}.
	\]
	Indeed, by the straightforward computation, $d_{n,v}(s,\ell)$ equals to a non-zero constant multiple of a confluent hypergeometric function $\xi(1,0;(s+\rho_n+\ell)/2,(s+\rho_n-\ell)/2)$.
	For the explicit formulas of $\xi$, see \cite{82_Shimura} and \cite[pp.\,140]{00_Shimura}.
	\end{rem}
	
	\subsection{Cuspidal components of Siegel Eisenstein series at $s=0$}
	
	We recall the properties of Siegel Eisenstein series at $s=0$.
	If the rank $n$ is odd and $\mu$ is quadratic, one has
	\[
	I_n(0, \mu) = \bigoplus_{V} R_n(V) \oplus \bigoplus_{\calC} R_n(\calC),
	\]
	where $V$ runs over all the quadratic spaces of dimension $n+1$ over $F$ such that $\mu = \chi_V$ and $\calC = \{W_v\}_v$ runs through all incoherent families such that $\mu_v = \chi_{W_v}$ for any place $v$ of $F$.
	For the definition of incoherent family, see \cite[pp.\,7]{1994_Kudla-Rallis}.
	By \cite[Theorem 4.10]{1994_Kudla-Rallis}, one can identify a certain subspace of automorphic forms as the space of Eisenstein series at $s=0$ as follows:
	
	\begin{thm}\label{s=0}
	The following statements hold.
	\begin{enumerate}
	\item For a quadratic space $V$ of dimension $n+1$ over $F$, one has
	\[
	\dim \mathrm{Hom}_{G_n(\bbA_\bbQ)} (R_n(V), \calA(G_n)) = 1.
	\]
	Moreover, the normalized Eisenstein series at $s = 0$ gives the non-trivial intertwining map $R_n(V) \longrightarrow \calA(G_n(\bbA_\bbQ))$.
	\item For an incoherent family $\calC$, one has
	\[
	\dim \mathrm{Hom}_{G_n(\bbA_\bbQ)} (R_n(\calC), \calA(G_n)) = 0.
	\]
	Moreover, for a standard section $f_s$ with $f_0 \in R_n(\calC)$, one has $E(g,0,f) = 0$.
	\end{enumerate} 
	\end{thm}
	
	The following statement follows from the theorem immediately.
	
	\begin{cor}\label{s=0_exp}
	Let $f_s$ be a standard section of $I_n(s,\mu)$.
	The candidates of real parts of non-zero cuspidal exponents of $E(\,\cdot\,,0,f)$ are only $((n-1)/2,(n-3)/2,\ldots,(1-n)/2)$.
	\end{cor}
	\begin{proof}
	By Theorem \ref{s=0}, the constant term of Eisenstein series along $B_n$ lies in the direct sum of induced representations of the form $I_n(0,\mu)$.
	The lemma then follows from $E(\,\cdot\, , s,f) \in \calA(G_n)_{\{B\}}$ and the definition of cuspidal exponents.
	\end{proof}

	\section{Pullback formula}
	
	In this section, we compute the pullback formulas of Siegel Eisenstein series.
	As an application, we show the holomorphy and non-vanishing of Klingen Eisenstein series.
	
	\subsection{The formal identity and meromorphic sections}
	
	For $m \leq n$, we define the embeddings $\iota_{m,n}^\uparrow$ and $\iota_{m,n}^\downarrow$ of $G_m$ into $G_n$ by
	\[
	\iota_{m,n}^\uparrow \left(\abcd \right) = \left(\begin{array}{cc|cc}a&&b&\\&\mathbf{1}_{n-m}&&\\\hline c&&d&\\&&&\mathbf{1}_{n-m} \end{array}\right), \qquad \iota_{m,n}^\downarrow \left(\abcd \right) = \left(\begin{array}{cc|cc}\mathbf{1}_{n-m}&&&\\&a&&b\\\hline &&\mathbf{1}_{n-m}&\\&c&&d \end{array}\right).
	\]
	Put $G_m^\uparrow = \iota_{m,n}^\uparrow (G_m)$ and $G_m^\downarrow = \iota_{m,n}^\downarrow (G_m)$.
	Take $n,r \in \bbZ_{>0}$.
	For $g \in G_{n+r} (\bbA_\bbQ)$ and $h \in G_{n}(\bbA_\bbQ)$, put
	\[
	g \times h =\begin{pmatrix}a_g & b_g \\ c_g & d_g \end{pmatrix} \times \begin{pmatrix}a_h & b_h \\ c_h & d_h \end{pmatrix} = \iota_{n+r,2n+r}^\uparrow(g)\cdot\iota_{n,2n+r}^\downarrow(h) = \left(\begin{array}{cc|cc} a_g & & b_g & \\ & a_h & & b_h \\ \hline c_g & & d_g & \\ & c_h & & d_h\end{array}\right) \in G_{2n+r}(\bbA_\bbQ).
	\]
	Set
	\[
	H = G_{n+r}^\uparrow \times G_{n}^\downarrow \subset G_{2n+r}, \qquad 
	\widehat{g} = \begin{pmatrix} 0 & \mathbf{1}_n \\ \mathbf{1}_n & 0 \end{pmatrix} g \begin{pmatrix} 0 & \mathbf{1}_n \\ \mathbf{1}_n & 0 \end{pmatrix}, \qquad g \in G_{n}.
	\]
	Let $f_s$ be a standard section of $I_{2n+r}(s,\mu)$.
	For a cusp form $\varphi$ on $G_{n}(\bbA_\bbQ)$ and $g \in G_{n+r}(\bbA_\bbQ)$, we consider the zeta integral
	\[
	E(g,s; f, \varphi) = \int_{G_{n}(\bbQ) \bs G_{n}(\bbA_\bbQ)} E((g \times \widehat{h}),s,f)\varphi(h) \, dh.
	\]
	
	Put
	\[
	f_j = \begin{pmatrix} 0 & 0 \\ 0 & \mathbf{1}_j\end{pmatrix} \in \Mat_{n+r,n}, \qquad \tau_j = \left(\begin{array}{cc|cc} \mathbf{1}_{n+r} & & & \\ & \mathbf{1}_n & & \\ \hline & f_j & \mathbf{1}_{n+r} & \\ {^t f_j} & & & \mathbf{1}_{n}\end{array}\right)
	\]
	for $0 \leq j \leq n$.
	Note that for any $g \in G_j(\bbA_\bbQ)$ and $h \in G_{2n+r}(\bbA_\bbQ)$, one has
	\[
	f_s\left(\tau_j((\mathbf{1}_{2(n+r-j)}\times g) \times (\mathbf{1}_{2(n-j)} \times \widehat{g})) h \right) = f_s(\tau_j h).
	\]
	The following double coset decomposition is well-known.
	For example, see \cite[Lemma 24.1]{00_Shimura}.
	
	\begin{lem}
	One has the decomposition
	\[
	G_{2n+r}(\bbQ) = \bigsqcup_{0 \leq j \leq n} P_{2n+r}(\bbQ) \tau_j H(\bbQ).
	\]
	Moreover, $P_{2n+r}(\bbQ)\tau_j H(\bbQ) = \coprod_{\xi, \beta, \gamma} P_{2n+r}(\bbQ)\tau_j(((\mathbf{1}_{2(n+r-j)} \times \xi) \times \mathbf{1}_{2n}) \cdot (\beta \times \gamma))$, where $\xi$ runs over $G_{j}(\bbQ)$, $\beta$ over $P_{n+r-j,n+r}(\bbQ) \bs G_{n+r}(\bbQ)$, and $\gamma$ over $P_{n-j,n}(\bbQ) \bs G_n(\bbQ)$.
\end{lem}

	By the lemma, we compute the integral $E(g,s;f,\varphi)$ as follows:
	\begin{align*}
	&\int_{G_n(\bbQ) \bs G_n(\bbA_\bbQ)} E((g \times \widehat{h}),s,f) \varphi(h) \, dh \\
	&= \int_{G_n(\bbQ) \bs G_n(\bbA_\bbQ)} \sum_{\gamma \in P_{2n+r}(\bbQ) \bs G_{2n+r}(\bbQ)} f_s(\gamma(g \times \widehat{h})) \varphi(h)\, dh \\
	&= \int_{G_n(\bbQ) \bs G_n(\bbA_\bbQ)} \sum_{0 \leq j \leq n} \sum_{\gamma \in P_{2n+r}(F) \bs P_{2n+r}(\bbQ) \tau_j H(\bbQ)} f_s(\gamma(g \times \widehat{h})) \varphi(h)\, dh\\
	&= \sum_{0 \leq j \leq n}  \int_{G_n(\bbQ) \bs G_n(\bbA_\bbQ)} \sum_{\xi \in G_j(\bbA_\bbQ)}\sum_{\beta \in P_{n+r-j,n+r}(\bbQ) \bs G_{n+r}(\bbQ)} \sum_{\gamma \in P_{n-j,n}(\bbQ) \bs G_{n}(\bbQ)} \\
	&\hspace{70mm} f_s(\tau_j ((\mathbf{1}_{2(n+r-j)} \times \xi) \times \mathbf{1}_{2n}) \cdot(\beta g \times \gamma \widehat{h})) \varphi(h) \, dh\\
	&= \sum_{0 \leq j \leq n} \sum_{\xi \in G_j(\bbA_\bbQ)}\sum_{\beta \in P_{n+r-j,n+r}(\bbQ) \bs G_{n+r}(\bbQ)} \int_{G_n(\bbQ) \bs G_n(\bbA_\bbQ)} \sum_{\gamma \in P_{n-j,n}(\bbQ) \bs G_n(\bbQ)}\\
	&\hspace{70mm} f_s(\tau_j((\mathbf{1}_{2(n+r-j)} \times \xi) \times \mathbf{1}_{2n}) \cdot (\beta g \times \gamma \widehat{h})) \varphi(h) \, dh.
	\end{align*}
	If $j < n$, we claim that the integral
	\[
	\int_{G_n(\bbQ) \bs G_n(\bbA_\bbQ)} \sum_{\gamma \in P_{n-j,n}(\bbQ) \bs G_n(\bbQ)}f_s(\tau_j((\mathbf{1}_{2(n+r-j)} \times \xi) \times \mathbf{1}_{2n}) \cdot (\beta g \times \gamma \widehat{h})) \varphi(h) \, dh
	\]
	vanishes.
	Put $P_{n-j,n}' = \{\widehat{p}\mid p \in P_{n-j,n}\}$.
	We write $\mathbf{1}_{2(n+r-j)} \times \xi$ by $\xi$ for simplicity.
	Then, it equals to
	\begin{align*}
	&\int_{P_{n-j,n}'(\bbQ) \bs G_n(\bbA_\bbQ)} f_s(\tau_j(\xi\beta g \times \widehat{h})) \varphi(h) \, dh\\
	&= \int_{P_{n-j,n}'(\bbQ)N_{P_{n-j,n}'}(\bbA_\bbQ) \bs G_n(\bbA_\bbQ)}\int_{N_{P'_{n-j,n}}(\bbQ) \bs N_{P'_{n-j,n}}(\bbA_\bbQ)} f_s(\tau_j(\xi\beta g \times \widehat{nh})) \varphi(nh) \,dn dh\\
	&= \int_{P_{n-j,n}'(\bbQ)N_{P_{n-j,n}'}(\bbA_\bbQ) \bs G_n(\bbA_\bbQ)}\int_{N_{P'_{n-j,n}}(\bbQ) \bs N_{P'_{n-j,n}}(\bbA_\bbQ)} f_s(\tau_j(\xi\beta g \times \widehat{h})) \varphi(nh) \, dndh\\
	&= \int_{P_{n-j,n}'(\bbQ)N_{P_{n-j,n}'}(\bbA_\bbQ) \bs G_n(\bbA_\bbQ)} f_s(\tau_j(\xi\beta g \times \gamma \widehat{h})) \left(\int_{N_{P'_{n-j,n}}(\bbQ) \bs N_{P'_{n-j,n}}(\bbA_\bbQ)} \varphi(nh) \, dn\right)dh\\
	&= 0.
	\end{align*}
	Hence, we obtain
	\begin{align*}
	E(g,s;f,\varphi) 
	&= \sum_{\xi \in G_{n}(\bbQ)}\sum_{\beta \in P_{r,n+r}(\bbQ) \bs G_{n+r}(\bbQ)} \int_{G_n(\bbQ) \bs G_n(\bbA_\bbQ)} f_s(\tau_n ( \mathbf{1}_{2(n+r)} \times \widehat{\xi}) \cdot(\beta g \times \widehat{h})) \varphi(h) \, dh\\
	&=  \sum_{\beta \in P_{r,n+r}(\bbQ) \bs G_{n+r}(\bbQ)} \int_{G_n(\bbQ) \bs G_n(\bbA_\bbQ)} \sum_{\xi \in G_n(\bbA_\bbQ)} f_s(\tau_n (\beta g \times \widehat{\xi h})) \varphi(h) \, dh\\
	&=  \sum_{\beta \in P_{r,n+r}(\bbQ) \bs G_{n+r}(\bbQ)} \int_{G_n(\bbA_\bbQ)} f_s(\tau_n (\beta g \times \widehat{h})) \varphi(h) \, dh.
	\end{align*}
	Put
	\[
	Z(g,s;f,\varphi) = \int_{G_{n}(\bbA_\bbQ)} f_s(\tau_n(g\times \widehat{h}))\varphi(h)\,dh, \qquad g \in G_{n+r}(\bbA_\bbQ).
	\]
	We then have 
	\[
	E(g,s;f,\varphi) = \sum_{\gamma \in P_{r,n+r}(\bbQ) \bs G_{n+r}(\bbQ)} Z(\gamma g,s;f,\varphi).
	\]
	
	\begin{lem}
	The integral $Z(g,s;f,\varphi)$ converges absolutely for $s \in \bbC$ with $\mathrm{Re}(s) \gg 0$ and can be meromorphically continued to whole $s$-plane.
	\end{lem}
	\begin{proof}
	Since $E(g,s,f)$ converges absolutely for $s$ with $\mathrm{Re}(s) \gg 0$, the integral also converges absolutely for such $s$.
	When $r=0$, the meromorphic continuation follows from the meromorphic continuation of $E(g,s,f)$.
	In general, we write $g = n(t,m)k$ for $n \in N_{P_{r,n+r}}(\bbA_\bbQ), (t,m) \in \GL_r(\bbA_F) \times G_n(\bbA_\bbQ)$ and $k \in K_{n+r}$.
	Then, one has
	\begin{align*}
	Z(g,s;f,\varphi) 
	&= \mu(t)|\det t|^{s+(2n+r+1)/2}Z(m,s;r(k)f,\varphi)\\
	&= \mu(t)|\det t|^{s+(2n+r+1)/2}Z(m,s+r/2;\iota_{2n,2n+r}^{\downarrow,*}r(k)f,\varphi).
	\end{align*}
	Thus, the meromorphic continuation follows from the case $r=0$.
	\end{proof}
	
	The section $Z(\,\cdot\,,s;f,\varphi)$ is then a meromorphic section of 
	\[
	I_{r,n+r}\left(s, \mu, \calA_{\mathrm{cusp}}(G_n)\right).
	\]
	Indeed, let $P$ be a parabolic subgroup of $G_{n}$ with the unipotent radical $N$.
	It suffices to prove that the constant term of $Z(\,\cdot\,,s;f,\varphi)$ along $P$ is zero.
	It equals to
	\begin{align*}
	Z(g,s;f,\varphi)_P
	&= \int_{N(\bbQ) \bs N(\bbA_\bbQ)} Z(ng,s;f,\varphi)\, dn\\
	&= \int_{N(\bbQ) \bs N(\bbA_\bbQ)}\int_{G_{n}(\bbA_\bbQ)} f_s(\tau_n(ng\times \widehat{h}))\varphi(h)\,dhdn\\
	&= \int_{G_{n}(\bbA_\bbQ)}\int_{N(\bbQ) \bs N(\bbA_\bbQ)} f_s(\tau_n(g\times \widehat{n^{-1}h}))\varphi(h)\,dndh\\
	&= \int_{G_{n}(\bbA_\bbQ)}f_s(\tau_n(g\times \widehat{h})) \left(\int_{N(\bbQ) \bs N(\bbA_\bbQ)} \varphi(nh)\,dn\right)dh\\
	&= 0,
	\end{align*}
	by the cuspidality of $\varphi$.
	Take a cusp form $\phi$ on $G_n(\bbA_\bbQ)$.
	For any $k \in K_{n+r}$, one has
	\begin{align*}
	\langle Z(k,s;f,\varphi), \phi\rangle
	&= \int_{G_n(\bbQ) \bs G_{n}(\bbA_\bbQ)} Z((\mathbf{1}_{r} \times x)k,s;f,\varphi) \overline{\phi(x)} \, dx\\
	&= \int_{G_n(\bbQ) \bs G_n(\bbA_\bbQ)} \int_{G_n(\bbA_\bbQ)} f_s(\tau_n (k \times \widehat{x^{-1}h})) \varphi(h) \, dh \, \overline{\phi(x)} \, dx\\
	&= \int_{G_n(\bbA_\bbQ)} f_s(\tau_n (k \times \widehat{h})) \left(\int_{G_n(\bbQ) \bs G_n(\bbA_\bbQ)} \varphi(xh) \overline{\phi(x)} \, dx\right) dh\\
	&= \int_{G_n(\bbA_\bbQ)} f_s(\tau_n (k \times \widehat{h})) \langle r(h)\varphi, \phi \rangle \, dh.
	\end{align*}
	The pairing $\langle Z(g,s;f,\varphi), \phi \rangle$ is zero unless $\phi$ lies in the $\pi_{\varphi}$-isotypic component of $\calA_{\mathrm{cusp}}(G_n)$.
	Here the representation $\pi_\varphi$ is the representation of $G_n(\bbA_\bbQ)$ generated by $\varphi$.
	For any $k \in K_{n+r}$, the function $m \longmapsto Z(mk,s;f,\varphi)$ on $G_n(\bbA_\bbQ)$ lies  in the $\pi_{\varphi}$-isotypic component.
	Hence, the section $Z(\,\cdot\,,s;f,\varphi)$ is a section of $I_{r,n+r}(s,\mu,\pi_\varphi)$.
	
	Let $\pi = \bigotimes_v\pi_v$ be an irreducible cuspidal automorphic representation of $G_n(\bbA_\bbQ)$.
	By the above computations, we define a meromorphic section $Z(\,\cdot\,,s;f,\varphi)$ of $I_{r,n+r}(s,\mu,\pi)$ for $\varphi \in \pi = \bigotimes_v\pi_v$.
	For $f_s = \bigotimes_v f_{v,s}$ and $\varphi = \bigotimes_v \varphi_v \in \bigotimes_v\pi_v$, set
	\[
	Z_v(g,s;f_v,\varphi_v) = \int_{\Sp_{2n}(F_v)} f_{v,s} (\tau_n(g \times \widehat{h})) \pi_v(h)\varphi_v\, dh.
	\]
	Then,
	\[
	Z(g,s;f,\varphi) = \prod_v Z_v(g,s;f_{v},\varphi_v).
	\]
	In the following, we first consider the relationship between the constant terms of $E(\,\cdot\,,s;f,\varphi)$ and the global section $Z(\,\cdot\,,s;f,\varphi)$.
	After that, we compute the local sections $Z(\,\cdot\,,s;f_v,\varphi_v)$.
	
	\subsection{Near holomorphy of Klingen Eisenstein series}
	
	We prove the near holomorphy of Eisenstein series $E(\,\cdot\,,s_0;f,\varphi)$ on $G_{n+r}(\bbA_\bbQ)$ as follows:
	
	\begin{prop}\label{NH_Klingen}
	Fix $r,n$ with $1 \leq r \leq n$ and $s_0 \geq 0$ with $s_0 \in \bbZ + (2n+r+1)/2$.
	For a character $\mu$ of $\GL_{2n+r}(\bbA_F)$, let $f_s$ be a standard section of $I_{2n+r}(s,\mu)$.
	We assume
	\begin{itemize}
	\item $f_{s_0}$ is $\frakp_{2n+r,-}$-finite.
	\item If $F = \bbQ$ and $s_0 = 1/2$, there exists a quadratic space $V$ over $F$ with dimension $(n+2)/2$ such that $W$ satisfies the condition \eqref{W} and $f_{s_0} \in R_n(V)$.
	\end{itemize}
	Then, for a cusp form $\varphi$ on $G_n(\bbA_\bbQ)$, the Eisenstein series $E(\,\cdot\,,s_0;f,\varphi)$ on $G_{n+r}(\bbA_\bbQ)$ is nearly holomorphic.
	\end{prop}
	\begin{proof}
	Under the assumptions, Siegel Eisenstein series $E(\,\cdot\,,s,f)$ is nearly holomorphic at $s=s_0$ by the proof of Lemma \ref{s>1} and Proposition \ref{s=1}.
	Take an integer $\ell \gg 0$ so that $\frakp_{2n+r,-}^\ell \cdot E(\,\cdot\,,s_0,f) = 0$.
	Since the integral
	\[
	E(g,s_0;f,\varphi) = \int_{G_{n+r}(\bbQ) \bs G_{n+r}(\bbA_\bbQ)} E((g \times \widehat{h}),s_0,f) \varphi(h) \, dh, \qquad g \in G_{n}(\bbA_\bbQ)
	\]
	converges absolutely, one has $\frakp_{n+r,-}^\ell \cdot E(g,s_0;f,\varphi) = 0$.
	This completes the proof.
	\end{proof}
	
	We next compute the constant term of $E(\,\cdot\,,s_0;f,\varphi)$ along $P_{r,n+r}$.
	Let $U$ be the subgroup of $G_{2n+r}$ in which elements of the form
	\[
	\left(\begin{array}{ccc|ccc} \mathbf{1}_{r} &  & & * & & \\ & \mathbf{1}_{n} &&&& \\ && \mathbf{1}_{n} &&& \\\hline &&&\mathbf{1}_{r} && \\ &&&& \mathbf{1}_n & \\ &&&&&\mathbf{1}_n \end{array}\right).
	\]
	We may regard the group $U$ as a subgroup of $G_{n+r}^\uparrow$.
	Then, it is a subgroup of the unipotent radical of $P_{r,n+r}$.
	Set
	\[
	E(g,s_0,f)_U = \int_{U(\bbQ) \bs U(\bbA_\bbQ)} E(ug,s_0,f)\, dn.
	\]
	We compute $E(g,s_0,f)_U$ as follows:
	
	\begin{lem}\label{lem_const_term_Klingen}
	Let $f_s$ be a standard section of $I_{2n+r}(s,\mu)$.
	Suppose that $f_s$ satisfies the conditions as in Proposition \ref{NH_Klingen} and moreover if $F=\bbQ$, assume $s_ 0 > 1$.
	We then have
	\[
	E((t,m),s_0,f)_U = \mu(t) |\det t|^{s_0 + (2n+r+1)/2}E(m,s_0+r/2,\iota_{2n,2n+r}^{\downarrow,*}f)
	\]
	for $(t,m) \in \GL_r(\bbA_F) \times G_{2n}(\bbA_\bbQ) = M_{P_{r,2n+r}}(\bbA_\bbQ)$.
	\end{lem}
	\begin{proof}
	By the near holomorphy of $E(g,s_0,f)$ and \cite[Lemma 5.10]{Horinaga_2}, we have
	\[
	E(g,s_0,f)_U = E(g,s_0,f)_{Q_{r,2n+r}}.
	\]
	Thus, for $(t,m) = (t_1,\ldots,t_r,m) \in \GL_1(\bbA_F) \times \cdots \times \GL_1(\bbA_F) \times  G_{2n}^\downarrow(\bbA_\bbQ) = M_{Q_{r,2n+r}}$, by taking the constant terms successively, we obtain
	\[
	E((t,m),s_0,f)_U = \left(\cdots \left( \left(E((t,m),s_0,f\right)_{Q_{1,2n+r}}|_{G_{2n+r-1}^\downarrow(\bbA_\bbQ)}\right)_{Q_{1,2n+r-1}}|_{G_{n+r-2}^\downarrow(\bbA_\bbQ)}\cdots\right)_{Q_{1,2n+1}}|_{G_{2n}^\downarrow(\bbA_\bbQ)}.
	\]
	We tacitly assume $r=1$.
	By (\ref{P_1_const_term}), one has
	\[
	E((t,m),s_0,f)_U = \mu(t)|t|^{s_0 + n+(r+1)/2} E(m,s_0+1/2,\iota^{\downarrow,*}f) + \mu(t)^{-1}|t|^{-s_0+n+1}E(m,s_0-1/2,\iota^{\downarrow,*}U(s,\mu)f).
	\]
	Then, for $s=s_0$ and $v \in \bfa$, the archimedean component $U_v(s,\mu)f_v$ has at least simple zero.
	Hence, by assumptions, the Eisenstein series $E(m,s-1/2,\iota^*U(s,\mu)f)$ is zero at $s=s_0$.
	For general $r$, we thus have
	\[
	E((t,m),s_0,f)_U = \left(\prod_{j = 1}^r\mu(t_j)|t_j|^{s_0 + (2n+r+1)/2}\right)E(m,s_0+r/2,\iota^{\downarrow,*}f).
	\]
	Let $\SL_r$ be the derived subgroup of $\GL_r \subset M_{P_{r,2n+r}}$.
	It suffices to show that $E(\,\cdot\,,s_0,f)$ is left $\SL_r(\bbA_F)$ invariant.
	It follows from \cite[Lemma 5.7]{Horinaga_2} by the near holomorphy of Eisenstein series.
	This completes the proof.
	\end{proof}
	\begin{prop}\label{const_term_Klingen}
	With the notation as in Proposition \ref{NH_Klingen}, suppose $s_0 > 1$ if $F = \bbQ$.
	Then, the constant term of $E(g,s_0;f,\varphi)$ along $P_{r,n+r}$ equals to the zeta integral $Z(g,s_0;f,\varphi)$ for any $g \in G_{n+r}(\bbA_\bbQ)$.
	\end{prop}
	\begin{proof}
	Since $K_{n+r,\infty}$ normalizes $\frakp_{n+r,-}$, the right translation $r(k)f_s$ satisfies the same condition as in Proposition \ref{NH_Klingen}.
	Thus, for any $(t,m) \in \GL_r(\bbA_\bbQ) \times G_{n} (\bbA_\bbQ) = M_{P_{r,n+r}}(\bbA_\bbQ)$ and $k \in K_{n+r}$, we have
	\begin{align*}
	E((t,m)k,s_0;f,\varphi)_{P_{r,n+r}} 
	&= \int_{U(\bbQ) \bs U(\bbA_\bbQ)}\int_{G_n(\bbQ) \bs G_n(\bbA_\bbQ)} E((u(t,m) \times \widehat{h}),s_0,r(k)f) \varphi(h) \, dhdu\\
	&= \mu(t)|\det t|^{s_0+\rho_{2n+r}}\int_{G_n(\bbQ) \bs G_n(\bbA_\bbQ)} E((m \times \widehat{h}),s_0 + r/2,\iota_{2n,2n+r}^{\downarrow,*}r(k)f) \varphi(h) \, dh \\
	&= \mu(t)|\det t|^{s_0+\rho_{2n+r}}Z(m,s_0 + r/2;\iota_{2n,2n+r}^{\downarrow,*}r(k)f,\varphi)\\
	&= Z((t,m),s_0;r(k)f,\varphi)\\
	&= Z((t,m)k,s_0;f,\varphi).
	\end{align*}
	For the first and second equality, we use Lemma \ref{lem_const_term_Klingen}.
	Hence we see $E(g,s_0;f,\varphi)_{P_{r,n+r}} = Z(g,s_0;f,\varphi)$.
	This completes the proof.
	\end{proof}
	
	\begin{cor}\label{hol_global_zeta_int}
	With the notation as in Proposition \ref{NH_Klingen}, suppose $s_0 > 1$ if $F=\bbQ$.
	Then, the zeta integral $Z(g,s;f,\varphi)$ is holomorphic at $s=s_0$.
	\end{cor}
	\begin{proof}
	The statement follows immediately from the definition of zeta integral and the holomorphy of $E(\,\cdot\,,s,f)$ at $s=s_0$.
	\end{proof}
	
	We next consider the local zeta integrals $Z_v(\,\cdot\,,s;f_v,\varphi_v)$.
	
	\subsection{Unramified computations}
	
	We first compute $Z_v(g,s;f,\varphi)$ at unramified places.
	\begin{lem}\label{zeta_int_unram}
	Let $\mu_v$ be an unramified character of $\GL_{2n+r}(F_v)$, $f_{v,s}$ be an unramified standard section of $I_{2n+r,v}(s,\mu_v)$ with $f_{v,s}(1) = 1$ and $\pi_v$ be an irreducible unramified representation of $\Sp_{2n}(F_v)$ with an invariant inner product $\langle\,,\,\rangle$.
	Take an unramified vector $\varphi_v \in \pi_v$ so that $\langle\varphi_v,\varphi_v\rangle = 1$.
	We then have
	\[
	Z(1,s;f_v,\varphi_v) = \frac{L_v(s+(r+1)/2,\pi_v,\mu_v)}{L_v(s+n+(r+1)/2,\mu_v)\prod_{j=1}^nL_v(2s+2n+r+1-2j,\mu^2_v)} \times \varphi_v.
	\]
	\end{lem}
	\begin{proof}
	The restriction $\iota^{\downarrow,*}f_{s+r/2}$ of $f_{v,s}$ to $G^\downarrow_{2n}$ is a standard unramified section of $I_{2n}(s+r/2,\mu)$.
	Since $Z(1,s;f_v,\varphi_v)$ is an unramified vector, it is a constant multiple of $\varphi_v$.
	By definition of local zeta integral, we have
	\begin{align*}
	\langle Z_v(1,s;f_v,\varphi_{v}), \varphi_v \rangle 
	&= \int_{\Sp_{2n}(F_v)} f_{v,s}(\tau_n(1\times \widehat{h})) \langle \pi_v(h)\varphi_v, \varphi_v\rangle \, dh\\
	&= \int_{\Sp_{2n}(F_v)} \iota_{2n,2n+r}^{\downarrow,*}f_{v,s+r/2}(\tau_n(1\times \widehat{h})) \langle \pi_v(h)\varphi_v, \varphi_v\rangle \, dh.
	\end{align*}
	By \cite[(7.2.8)]{1994_Kudla-Rallis}, one has
	\[
	Z(1,s;f_v,\varphi_v) = \frac{L_v(s+(r+1)/2,\pi_v,\mu_v)}{L_v(s+n+(r+1)/2,\mu_v)\prod_{j=1}^nL_v(2s+2n+r+1-2j,\mu^2_v)}\times\varphi_v.
	\]
	This completes the proof.
	\end{proof}
	
	\subsection{Computations of ramified places}
	
	Fix a non-archimedean place $v$.
	In this subsection, we compute the zeta integrals at the non-archimedean ramified place $v$.
	We then show the following lemma.
	
	\begin{lem}\label{zeta_int_ram}
	Let $\alpha_s$ be a standard section of $I_{r,n+r,v}(s,\mu_v,\pi_v)$.
	There exists a finite number of standard sections $f_{v,s,1},\ldots,f_{v,s,\ell}$ of $I_{2n+r,v}(s,\mu)$ and vectors $\varphi_{v,1}, \ldots, \varphi_{v,\ell} \in \pi_v$ such that
	\[
	\sum_{j = 1}^\ell Z_v(g,s;f_{v,j},\varphi_{v,j}) = \alpha_s(g), \qquad g \in \Sp_{2n}(F_v).
	\] 
	\end{lem}
	\begin{proof}
	Put
	\[
	K_{n,v}(\frakp_v^a) = \{k \in K_{n,v} \mid \text{$k \equiv \mathbf{1}_n$ mod $\frakp_v^a$}\}.
	\]
	Let $\ell$ be a positive integer such that $\alpha_s$ is fixed by $K_{n+r,v}(\frakp_v^\ell)$.
	We write $K = K_{n+r,v}(\frakp_v^\ell)$.
	Let $\{\gamma_1,\ldots,\gamma_\ell\} \subset K_{n+r,v}$ be a set of complete representatives of $P_{r,n+r}(F_v) \bs \Sp_{2(n+r)}(F_v)/K$.
	We may assume $\gamma_1 =1$.
	Put $\varphi_j = \alpha_s(\gamma_j)$ and $K_{\varphi_j} = \mathrm{Stab}_{K_{n,v}}(\varphi_j)$.
	We claim that for any $j$, one has $\mathrm{pr}_n(K_{n+r,v} \cap P_{r,n+r}(F_v)) \subset K_{\varphi_j}$.
	Here, $\mathrm{pr}_n$ is the projection map $\mathrm{pr}_n \colon P_{r,n+r}(F_v) \longrightarrow \GL_r \times \Sp_{2n} \longrightarrow \Sp_{2n}$.
	Indeed, take $k \in \mathrm{pr}_n(K_{n+r,v} \cap P_{r,n+r}(F_v))$.
	Fix $k' \in K$ such that $\mathrm{pr}_n(k') = k$.
	By the choice of $K$, one has $\pi_v(k) \varphi_j = \alpha_s(k' \gamma_j)$.
	Since $K$ is a normal subgroup of $K_{n+r,v}$, one obtains $\alpha_s(k'\gamma_j) = \alpha_s(\gamma_j\gamma_j^{-1}k'\gamma_j) = \alpha_s(\gamma_j)$.
	Thus, $\pi_v(k)\varphi_j = \varphi_j$ and $k \in K_{\varphi_j}$.
	
	Let $f_{v,s,j}$ be a standard section of $I_{2n+r}(s,\mu)$ such that
	\begin{itemize}
	\item $\mathrm{supp}(f_{v,s,j}) \subset P_{2n+r}(F)\tau_n (K \times K_{\varphi_j}')$.
	\item $f_{v,s,j}(p\tau_n(k_1 \times k_2)) = \frac{1}{\mathrm{vol}(K_{\varphi_j})}\mu(p)|p|^{s+(2n+r+1)/2}$ for $p \in P_{2n+r}(F_v)$ and $k_1  \times k_2 \in K \times K_{\varphi_j}$.
	\end{itemize}
	Here, $K_{\varphi_j}' = \{\widehat{k} \mid k \in K_{\varphi_j}\}$.
	Let $k \in K$.
	By the claim, if $\tau(k \times \widehat{h}) \in \mathrm{supp}(f_{v,s,j})$, one has $h \in K_{\varphi_j}$.
	Thus, we have
	\begin{align*}
	Z_v(k,s;f_{v,j},\varphi_{j}) 
	&= \int_{\Sp_{2n}(F_v)} f_{v,s,j}(\tau_n(k \times \widehat{h})) \pi_v(h) (\varphi_j) \, dh\\
	&= \int_{K_{\varphi_j}} f_{v,s,j}(\tau_n(k \times \widehat{h})) \pi_v(h)(\varphi_j) \, dh\\
	&= \varphi_j.
	\end{align*}
	
	Next, we compute the support of the section.
	For $g \in \Sp_{2(n+r)}(F_v)$, we assume $Z_v(g,s; f, \varphi_j) \neq 0$.
	Suppose that $g$ lies in $P_{r,n+r}(F_v)\gamma_q K_{q}$ for some $q \neq 1$.
	Then, by the definition of $f_{v,s,j}$, one has
	\[
	f_{v,s,j}(\tau_n(g \times \widehat{h})) = f_{v,s,j}(\tau_n(h^{-1}g\times 1))
	\]
	for any $h$.
	By $h^{-1}g \in P_{r,n+r}(F_v)\gamma_q K$ with $q \neq 1$, we get $f_{v,s,j}(\tau_n(g \times \widehat{h})) = 0$.
	Hence we obtain
	\[
	Z_v(g,s;f_{v,j},\varphi_j) = 0
	\]
	and $\supp(r(\gamma_j^{-1})Z_v(\,\cdot\,,s;f_{v,j},\varphi_j)) = P_{r,n+r}(F_v)K\gamma_j = P_{r,n+r}\gamma_jK$.
	We then have
	\[
	\alpha_s(g) = \sum_{j=1}^\ell r(\gamma_j^{-1})Z_v(g,s;f,\varphi_j).
	\]
	This completes the proof.
	\end{proof}
	
	\subsection{Computations of archimedean places}
	In this subsection, we assume $F = \bbQ$ for simplicity.
	Let $v$ be the archimedean place of $F = \bbQ$.
	Let $\pi$ be a holomorphic discrete series representation of $G_n(\bbR)$ with highest weight $\lambda = (\lambda_{1,v},\ldots,\lambda_{n,v})_v$.
	For a standard section $f_s$ of $I_{2n+r}(s,\mu)$, put
	\[
	Z_v(g,s;f,\varphi,\varphi') = \langle Z_v(g,s;f,\varphi), \varphi'\rangle
	\]
	for $g \in G_{n+r}(F_v)$ and $\varphi' \in \pi$.
	Here, $\langle\,\cdot\,,\,\cdot\,\rangle$ is an invariant inner product on $\pi$.
	
	\begin{lem}\label{abs_conv_non_zero}
	With the above notation, suppose that a real number $s_0$ satisfies $s_0 \in \bbZ + \rho_{2n+r}$ and $-r/2 < s_0 \leq \lambda_{n,v} - \rho_{2n+r}$ for any $v \in \bfa$.
	Let $f_s$ be a standard section of $I_{2n+r}(s,\mu)$ such that $f_{s_0}$ is $\frakp_{2n+r,-}$-finite.
	Then, the integral $Z_v(g,s;f,\varphi,\varphi')$ converges absolutely at $s=s_0$ for any $g \in G_{n+r}(F_v)$ and $v,v' \in \pi$.
	Moreover, we may choose $g,f_s$ and $\varphi,\varphi' \in \pi$ so that $Z_v(g,f;s,\varphi,\varphi')$ is non-zero at $s=s_0$.
	\end{lem}
	\begin{proof}
	For $m \in G_{n}(F_v)$, one has
	\[
	Z_v((\mathbf{1}_r \times m)k,s;f,\varphi,\varphi') = Z_v(1,s;r(k)f,\varphi,\pi(m^{-1})\varphi').
	\]
	Since the standard section $r(k)f_s$ satisfies the assumption as in the statement, we may assume $g = 1$.
	Then, the integral equals to
	\[
	Z_v(1,s;f,v,v') = \int_{G_n(F_v)} f_s(\tau_n(\mathbf{1}_{n+r} \times \widehat{h})) \langle \pi(h)v, v'\rangle \, dh = \int_{G_n(F_v)} f_s(\tau_n((\mathbf{1}_{r}\times h) \times \mathbf{1}_n)) \langle \pi(h)v, v'\rangle \, dh.
	\]
	Consider the restriction of $f_s$ to the subgroup $G_{2n}^\downarrow(F_v)$.
	The restriction $\iota_{2n,2n+r}^{*,\downarrow}f_s$ to $G^\downarrow_{2n}(F_v)$ is a standard section of $I_{2n}(s + r/2,\mu)$.
	The restriction map $\iota_{2n,2n+r}^{\downarrow,*}$ induces a non-zero intertwining map
	\[
	I_{2n+r}(s_0,\mu)_{\frakp_{2n+r,-}\fin} \longrightarrow I_{2n}(s_0+r/2,\mu)_{\frakp_{2n,-}\fin} = L(m/2,\ldots,m/2).
	\]
	Thus, $Z_v$ induces 
	\[
	L(m/2,\ldots,m/2) \otimes \pi \otimes \pi \longrightarrow \bbC.
	\]
	Here, $m = 2s_0 +2n + r + 1 \geq 2n+2$.
	This map is the same as in \cite[(4.3.4)]{2020_Liu}.
	Hence, the lemma follows from \cite[Proposition 4.3.1]{2020_Liu}.
	This completes the proof.
	\end{proof}
	
	\begin{cor}\label{zeta_int_arch}
	With the above notation, the zeta integral at $s=s_0$ induces a non-zero intertwining map
	\[
	I_{2n+r}(\mu,s_0)_{\frakp_{2n+r,-}\fin} \otimes \pi \longrightarrow I_{r,n+r}(s_0,\mu,\pi)_{\frakp_{n+r,-}\fin}.
	\]
	\end{cor}
	\begin{proof}
	By Lemma \ref{abs_conv_non_zero}, the zeta integral defines a non-zero intertwining map
	\[
	I_{2n+r}(\mu,s_0)_{\frakp_{2n+r,-}\fin} \otimes \pi \longrightarrow I_{r,n+r}(s_0,\mu,\pi).
	\]
	Since the integral is intertwining, the image is contained in the $\frakp_{n+r,-}$-finite vectors.
	This completes the proof.
	\end{proof}

	\section{Structure theorem of the space of nearly holomorphic automorphic forms}
	
	In this section, we compare the space of nearly holomorphic automorphic forms with the space of Eisenstein series.

	\subsection{Parametrization of infinitesimal characters}
	
	For an infinitesimal character $\chi$ of $\calZ_n$, put
	\[
	\calN(G_n,\chi) = \{\varphi \in \calN(G_n) \mid \text{$(z-\chi(z))\varphi = 0$ for any $z \in \calZ_n$}\}.
	\]
	By \cite[Proposition 5.15]{Horinaga_2}, we have
	\begin{align}\label{ss_Z_n}
	\calN(G_n) = \bigoplus_\chi \calN(G_n,\chi),
	\end{align}
	where $\chi$ runs over all integral infinitesimal characters of $\calZ_n$.
	We define $\calN(G_n,\chi)_{\{P\}}$ and $\calN(G_n,\chi)_{(M,\pi)}$ similarly.
	By \cite[Proposition 5.9]{Horinaga_2}, the constant term along $Q_{i,n}$ induces an embedding of the space $\calN(G_n,\chi)_{(M_{Q_{i,n}},\mu \boxtimes \pi)}$ into the direct sum
	\[
	\bigoplus_{s_0} I_{i,n}(s_0, \mu,\pi).
	\]
	Here $s_0$ runs over all real numbers such that the induced representation $I_{i,n}(s_0, \mu,\pi)$ has the integral infinitesimal character $\chi$.
	Take a real number $t$.
	We define the projection ${\mathbf{pr}}_t$ by
	\[
	\mathbf{pr}_{t} \colon \bigoplus_{s_0} I_{i,n}(s_0, \mu,\pi) \longrightarrow I_{i,n}(t,\mu,\pi).
	\]
	The infinitesimal character of the induced representation $I_{i,n}(s_0, \mu,\pi)$ has the Harish-Chandra parameter
	\[
	(\lambda_{1,v}, \ldots, \lambda_{n-i,v}, s_0 + n - (i-1)/2, \ldots, s_0 + n -(i-1)/2) + \rho.
	\]
	Here, $(\lambda_{1,v}, \ldots,\lambda_{n-i,v})$ is the highest weight of $\pi_v$ for any $v\in \bfa$.
	For $\chi_{s_0}$, we mean the infinitesimal character of the induced representation.
	Note that $\chi_{s_0}$ depends on $\lambda$ and $i$.

	\begin{lem}
	With the above notation, fix $s_0$.
	Let $\{t_1,\ldots,t_\ell\}$ be the set of real numbers such that $\mathbf{pr}_{t_j}(\varphi_{Q_{i,n}}) \neq 0$ for some $\varphi \in \calN(G_n,\chi_{s_0})_{(\mu\boxtimes\pi,M_{Q_{i,n}})}$.
	Then, for any $j$, the highest weight submodule of $I_{i,n}(t_j, \mu,\pi)$ is unitarizable.
	\end{lem}
	\begin{proof}
	We may assume $t_1 < \cdots < t_\ell \leq \rho_{i,n} + s_0$.
	Note that the highest weight of $I_{i,n}(t_j, \mu,\pi)$ is of the form $a_j = (\lambda_{1,v},\ldots,\lambda_{n-i,v}, \rho_{i,n} + t_j,\ldots,\rho_{i,n} + t_j)$.
	Then, $a_1$ is maximal in $\{a_1,\ldots,a_\ell\}$.
	By assumption, there exists $\varphi \in \calN(G_n,\chi_{s_0})_{(\mu\boxtimes\pi,M_{Q_{i,n}})}$ such that $\varphi$ is of weight $a_1$.
	Note that by maximality of $a_1$, for $j \neq 1$, the $K_{n,\infty}$-type $\rho_{a_1}$ does not occur in $I_{i,n}(t_j,\mu,\pi)_{\frakp_{n,-}\fin}$.
	Then, $\varphi_{Q_{i,n}}$ lies in $I_{i,n}(t_1, \mu,\pi)$.
	By \cite[Corollary 7.3]{Horinaga_2}, the module generated by $\varphi_{Q_{i,n}}$ is isomorphic to $L(a_1)$.
	By \cite[I.4.11]{MW}, if $t_1 < 0$, the automorphic form $\varphi$ is of square-integrable.
	Thus, the highest weight module $L(a_1)$ is unitarizable.
	If $t_1 \geq 0$, the highest weight module $L(a_1)$ is unitarizable by Theorem \ref{unitary}.
	Since the highest weight submodule of $I_{i,n}(t_j, \mu,\pi)$ is irreducible with integral weight $a_j$, the highest weight submodules $L(a_j)$ are unitarizable by Theorem \ref{unitary} for all $j$.
	This completes the proof.
	\end{proof}
	
	If the induced representation $I_{i,n}(s_0, \mu,\pi)$ contains a unitary highest weight representation, one has $s_0 \in \bbZ + n - (i-1)/2$ with $n-i \leq s_0 + n - (i-1)/2 \leq \lambda_{n,v}$.
	The following statement follows from the straightforward computation.
	For details, see \cite[Proposition 6.4]{Horinaga_2}.
	
	\begin{lem}\label{s_0}
	With the above notation, suppose for simplicity $F=\bbQ$.
	Let $a,b$ be real numbers so that $a,b \in \bbZ + n - (i-1)/2, n-i \leq a + n - (i-1)/2 \leq \lambda_{n,v}$ and $n-i \leq b + n - (i-1)/2 \leq \lambda_{n,v}$.
	Then, one has $\chi_a = \chi_b$ if and only if $|a| = |b|$.
	\end{lem}
	
	Put
	\[
	\calN^2(G_n,\chi)_{(M,\pi)} = \{\varphi \in \calN(G_n,\chi)_{(M,\pi)} \mid \text{$\varphi$ is square-integrable}\}.
	\]
	In the following of this section, we study $\calN(G_n,\chi)_{(M,\pi)}$ in terms of $\calN^2(G_n,\chi)_{(M,\pi)}$ and induced representations.
	
	\subsection{Constant terms of nearly holomorphic automorphic forms}
	
	Toward the classification of $(\frakg_n,K_{n,\infty})$-modules generated by nearly holomorphic automorphic forms on $G_n(\bbA_\bbQ)$, we investigate the embedding of $\calN(G_n)_{(M,\pi)}$ into a direct sum of induced representations.
	Fix a positive integer $i \leq n$.
	Let $\mu$ be a character of $\GL_1(\bbA_F)$ and $\pi$ an irreducible holomorphic cuspidal automorphic representation on $G_{n-i}(\bbA_\bbQ)$ with $\pi_v = L(\lambda_v) = L(\lambda_{1,v},\ldots,\lambda_{n-i,v})$ for $v \in \bfa$.
	Put $\Pi = \mu \boxtimes \pi$.
	For the notation, see \S \ref{NHAF}.
	We consider the space $\calN(G_n,\chi_{s_0})_{(M_{Q_{i,n}}, \Pi)}$.
	By Lemma \ref{s_0}, the constant term along $P_{i,n}$ induces the embedding
	\begin{align}\label{emb_const_term}
	\calN(G_n,\chi_{s_0})_{(M_{Q_{i,n}}, \Pi)} \xhookrightarrow{\,\quad\,} 
	\begin{dcases}
	\left(I_{i,n}(-s_0, \mu,\pi) \oplus I_{i,n}(s_0, \mu,\pi)\right)_{\frakp_{n,-}\fin}&\text{if $s_0 \neq 0$}\\
	I_{i,n}(s_0, \mu,\pi)_{\frakp_{n,-}\fin}&\text{if $s_0 = 0$}.
	\end{dcases}
	\end{align}
	Note that $\calN(G_n,\chi_{s_0})_{(M_{Q_{i,n}}, \Pi)} = \calN^2(G_n,\chi_{s_0})_{(M_{Q_{i,n}}, \Pi)}$ if and only if the image of (\ref{emb_const_term}) is contained in $I_{i,n}(-s_0, \mu,\pi)$.
	In this case, the space $\calN(G_n,\chi_{s_0})_{(M_{Q_{i,n}}, \Pi)}$ is semisimple as $(\frakg_n,K_{n,\infty})$-modules.
	The highest weights of the right hand side of (\ref{emb_const_term}) are of the form
	\[
	(\lambda_{1,v},\ldots,\lambda_{n-i,v}, \rho_{i,n}+s_0,\ldots,\rho_{i,n}+s_0)_v
	\]
	and
	\[
	(\lambda_{1,v},\ldots,\lambda_{n-i,v}, \rho_{i,n}-s_0,\ldots,\rho_{i,n}-s_0)_v
	\]
	if exist.
	If $\lambda_{n-i,v} < \rho_{i,n}$ for some $v$, one has $\calN(G_n,\chi_{s_0})_{(M_{Q_{i,n}}, \Pi)} = \calN^2(G_n,\chi_{s_0})_{(M_{Q_{i,n}}, \Pi)}$.
	Thus, for the classification, it suffices to consider the case where $\lambda$ satisfies $\lambda_{n-i,v} \geq \rho_{i,n}$ for any $v \in \bfa$.
	In this case, we may assume $0\leq s_0 \leq \mathrm{min}_{v\in \bfa}\{\lambda_{n-i,v} - \rho_{i,n}\}$.
	
	\begin{lem}\label{isotypic_s=0}
	Under the above assumption, if $s_0 = 0$, the space $\calN(G_n,\chi_{s_0})_{(M_{Q_{i,n}}, \Pi)}$ is isotypic for
	\[
	\boxtimes_{v\in \bfa} L(\lambda_{1,v},\ldots\lambda_{n-i,v}, \rho_{i,n},\ldots,\rho_{i,n})
	\]
	as $(\frakg_n,K_{n,\infty})$-modules.
	\end{lem}
	\begin{proof}
	By (\ref{emb_const_term}), one has
	\[
	\calN(G_n,\chi_{s_0})_{(M_{Q_{i,n}}, \Pi)} \xhookrightarrow{\,\quad\,} I_{i,n}(0,\mu,\pi).
	\]
	Consider the induced representation
	\[
	I_{i,n,v}(0, \mu_v, L(\lambda_v))
	\]
	for $v \in \bfa$ and a unitary character $\mu_v$.
	Since this induced representation lies in the unitary axis, it is unitary by the unitarizability of $L(\lambda)$.
	Thus, it is semisimple as $(\frakg_n,K_{n,\infty})$-modules.
	Highest weights in it are of the form $(\lambda_{1,v},\ldots\lambda_{n-i,v}, \rho_{i,n},\ldots,\rho_{i,n})$.
	We then have
	\[
	I_{i,n,v}(0,\mu_v,L(\lambda_v))_{\frakp_{n,-}\fin} \subset L(\lambda_{1,v},\ldots\lambda_{n-i,v}, \rho_{i,n},\ldots,\rho_{i,n}).
	\]
	This completes the proof.
	\end{proof}
	
	In the following of this section, we assume $\lambda_{n-i,v} > \rho_{i,n}$ for any $v\in \bfa$, $s_0 \in \bbZ + \rho_{i,n}$ and $0 < s_0 \leq \mathrm{min}_{v \in \bfa}\{\lambda_{n-i,v}-\rho_{i,n}\}$.
	We then have
	\[
	\calN^2(G_n,\chi_{s_0})_{(Q_{i,n},\mu\boxtimes\pi)} \bs \calN(G_n,\chi_{s_0})_{(Q_{i,n},\mu\boxtimes\pi)} \xhookrightarrow{\qquad} I_{i,n}(s_0,\mu,\pi)_{\frakp_{n,-}\fin}.
	\]

	\subsection{Structure theorem for $i=n$}
	\begin{prop}\label{F neq Q, P = B}
	We assume that either of the following conditions holds:
	\begin{itemize}
	\item $F \neq \bbQ$ and $s_0 > 0$.
	\item $\mu^2 \neq \mathbf{1}$ and $s_0 > 0$.
	\end{itemize}
	We then have
	\[
	\calN(G_{n}, \chi_{s_0})_{(B,\mu)} \cong \calN^2(G_{n}, \chi_{s_0})_{(B,\mu)} \oplus I_n(s_0, \mu)_{\frakp_{n,-}\fin}.
	\]
	\end{prop}
	\begin{proof}
	By \eqref{emb_const_term}, the constant term along $P_n$ induces the injective map
	\[
	\calN^2(G_{n}, \chi_{s_0})_{(B,\mu)} \bs \calN(G_{n}, \chi_{s_0})_{(B,\mu)} \xhookrightarrow{\qquad} I_n(s_0, \mu)_{\frakp_{n,-}\fin}.
	\]
	By Lemma \ref{s>1}, the Eisenstein series at $s = s_0$ gives the splitting
	\[
	I_n(s_0, \mu)_{\frakp_{n,-}\fin} \xhookrightarrow{\qquad} \calN(G_{n},\chi_{s_0})_{(B,\mu)}.
	\]
	Hence the statement follows.
	\end{proof}
	
	Next we treat the case $F = \bbQ$.
	
	\begin{prop}\label{F = Q, P = B}
	The following statements hold.
	\begin{enumerate}
	\item For $s_0 > 1$, one has
	\[
	\calN(G_{n},\chi_{s_0})_{(B,\mu)} \cong \calN^2(G_{n},\chi_{s_0})_{(B,\mu)} \oplus I_n(s_0, \mu)_{\frakp_{n,-}\fin}.
	\]
	\item For $s_0 = 1$, one has
	\[
	\calN^2(G_{n},\chi_{s_0})_{(B,\mu)} \bs \calN(G_{n},\chi_{s_0})_{(B,\mu)} \cong I_n(1,\mu)_{\frakp_{n,-}\fin}.
	\]
	Moreover, there are no splitting $I_{n}(1,\mu)_{\frakp_{n,-}\fin} \longrightarrow \calN(G_n,\chi_{s_0})_{(B,\mu)}$ if $I_{n}(1,\mu)_{\frakp_{n,-}\fin} \neq 0$.
	\item For $s_0 = 1/2$, one has
	\[
	\calN(G_{n},\chi_{s_0})_{(B,\mu)} = \calN^2(G_{n},\chi_{s_0})_{(B,\mu)}, \qquad \text{if $\mu_v \neq \mathrm{sgn}^{(n+2)/2}$ for any $v \in \bfa$}
	\]
	and
	\[
	\calN(G_{n},\chi_{s_0})_{(B,\mu)} \subset I_{n}(1/2,\mu)_{\frakp_{n,-}\fin}, \qquad \text{if $\mu_v = \mathrm{sgn}^{(n+2)/2}$ for any $v \in \bfa$}.
	\]
	\end{enumerate}
	\end{prop}
	\begin{proof}
	The proof of (1) is the same as the proof of Proposition \ref{F neq Q, P = B}.
	
	Next we show (2).
	If $\mu_v \neq \mathrm{sgn}^{(n+3)/2}$ for some $v \in \bfa$, one has $\calN(G_{n},\chi_{s_0})_{(B,\mu)} = 0$ and $I_n(1,\mu)_{\frakp_{n,-}\fin} = 0$.
	We may assume $\mu_v = \mathrm{sgn}^{(n+3)/2}$ for any $v \in \bfa$.
	Take $f = \bigotimes_v f_v \in I_n (1,\mu)$ such that $f_v$ lies in $R_n(W_v)$ for some $W_v$ and $f \not\in \sum_V R_n(V)$.
	Here $V$ runs over all positive definite quadratic forms over $F$ of dimension $n+3$.
	Let $f_s$ be the standard section of $I_n(s,\mu)$ such that $f_1 = f$.
	We assume that there exists a nearly holomorphic automorphic form $\varphi \in \calN(G_n)_{\{B\}}$ such that $\varphi_{P_n} = f$.
	By Proposition \ref{s=1}, the difference $\varphi - E(\,\cdot\,,1,f)$ is non-zero and square integrable.
	However, for $v\in\bfa$, the $K_{n,v}$-type $((n+3)/2, \ldots, (n+3)/2)$ in $I_{n,v}(-1,\mu_v)$ generates a reducible indecomposable representation of $\Sp_{2n}(F_v)$.
	This contradicts to the square integrability.
	Hence there are no automorphic form $\varphi$ such that $\varphi_{P_n} = f$.
	Recall that the constant term along $P_n$ induces the inclusion
	\[
	\calN^2(G_n,\chi_{s_0})_{(B,\mu)} \bs \calN(G_n,\chi_{s_0})_{(B,\mu)} \xhookrightarrow{\,\quad\,} I_n(1,\mu)_{\frakp_{n,-}\fin}.
	\]
	The image of $E(\,\cdot\,,1,f)$ is the same as $f$.
	Hence the above inclusion is surjective and there are no splitting.
	This completes the proof of (2).
	
	For (3), we assume $\mu_v \neq \mathrm{sgn}^{(n+2)/2}$ for any $v \in \bfa$.
	Note that if $I_{n,v}(1/2,\mu_v)$ has a highest weight vector, one has $\mu_v = \mathrm{sgn}^{(n+2)/2}$.
	Thus, the constant term (\ref{emb_const_term}) induces the embedding
	\[
	\calN(G_n,\chi_{s_0})_{(B,\mu)} \xhookrightarrow{\,\quad\,} I_n(-1/2,\mu)_{\frakp_{n,-}\fin}.
	\]
	We then have $\calN(G_n,\chi_{s_0})_{(B,\mu)} = \calN^2(G_n,\chi_{s_0})_{(B,\mu)}$.
	The last statement follows immediately from \eqref{emb_const_term}.
	This completes the proof.
	\end{proof}
	
	\subsection{Structure theorem for $P \neq B$}
	Fix $i$.
	We consider the case $P = P_{i,n}$
	Let $\mu$ be a character of $\GL_i(\bbA_F)$ and $\pi$ an irreducible holomorphic cuspidal representation of $G_{n-i}(\bbA_\bbQ)$ and $s_0 \in \bbZ + \rho_{i,n}$.
	Suppose $s_0 > 0$.
	Let $S$ be a finite set of places such that $\bfa \subset S$ and for $v \not\in S$, the representations $\mu_v$ and $\pi_v$ are unramified.
	Set
	\[
	L^S(s,\pi,\mu) = \prod_{v \not\in S} L(s,\pi_v,\mu_v).
	\]
	
	\begin{lem}\label{choice_global}
	Let $\alpha = \bigotimes_v \alpha_v \in I_{i,n}(s_0,\mu,\pi)_{\frakp_{n,-}\fin}$ and $S$ the finite set of places such that for $v \not\in S$, the function $\alpha_v$ is unramified.
	Then, there exists finite number of standard sections $f_1,\ldots,f_\ell$ of $I_{2n+r}(s,\mu)$ and $\varphi_1,\ldots\varphi_\ell \in \pi$ such that
	\[
	\lim_{s \rightarrow s_0} \frac{1}{L^S(s + (r+1)/2,\pi,\mu)}\sum_{j=1}^\ell Z(g,s; f_j, \varphi_j) = \alpha(g), \qquad g \in G_n(\bbA_\bbQ).
	\]
	\end{lem}
	\begin{proof}
	The statement follows from Lemma \ref{zeta_int_unram}, Lemma \ref{zeta_int_ram} and Corollary \ref{zeta_int_arch}.
	\end{proof}
	
	\begin{prop}\label{general_P_neq_B}
	Suppose $s_0 > 1$ if $F = \bbQ$.
	We then have
	\[
	\calN(G_{n},\chi_{s_0})_{(M_{Q_{i,n}}, \mu\boxtimes\pi)} \cong \calN^2(G_{n},\chi_{s_0})_{(M_{Q_{i,n}}, \mu\boxtimes\pi)} \oplus I_{i,n}(s_0,\mu,\pi)_{\frakp_{n,-}\fin}.
	\]
	\end{prop}
	\begin{proof}
	It suffices to show that for any $\frakp_{n,-}$-finite function $\alpha \in I_{i,n}(s_0,\mu,\pi)$, there exists a nearly holomorphic automorphic form $\varphi \in \calN(G_n,\chi_{s_0})_{(M_{Q_{i,n}}, \mu\boxtimes\pi)}$ such that $\varphi_{P_{i,n}}= \alpha$.
	This follows immediately from Proposition \ref{const_term_Klingen} and Lemma \ref{choice_global}.
	This completes the proof.
	\end{proof}
	
	In the following, we give partial results.
	
	\begin{prop}\label{P neq B}
	Assume $F = \bbQ$.
	Let $\Pi = \mu \boxtimes \pi$ be an irreducible holomorphic cuspidal automorphic representation of $M_{P_{i,n}}(\bbA_\bbQ)$.
	Suppose that highest weights of the archimedean component $\bigotimes_{v \in \bfa} \pi_v = \bigotimes_{v \in \bfa} L(\lambda_{1,v}, \ldots,\lambda_{n-i,v})$ satisfies $\lambda_{n-i,v} \geq \rho_{i,n} + s_0$.
	We then obtain the following result:
	\begin{enumerate}
	\item For $s_0 = 1/2$, the space $\calN(G_n,\chi_{s_0})_{(M,\Pi)}$ is $\bigotimes_{v \in \bfa} L(\lambda_{1,v},\ldots,\lambda_{n-i,v},\rho_{i,n} + \vep,\ldots,\rho_{i,n}+\vep)$-isotypic.
	Here, $\vep \in \{\pm 1/2\}$ is defined so that $\mathrm{sgn}^{\rho_{i,n}+\vep} = \mu_v$ for any $v$.
	\item For $s_0 = 1$, the space $\calN(G_n,\chi_{s_0})_{(M_{Q_{i,n}},\Pi)}$ is contained in
	\[
	I_{i,n}(-1, \mu,\pi) \oplus I_{i,n}(1, \mu,\pi).
	\]
	\end{enumerate}
	\end{prop}
	\begin{proof}
	The statements follow from (\ref{emb_const_term}) immediately.
	\end{proof}
	
	\subsection{Classification of $(\frakg_n,K_{n,\infty})$-module generated by nearly holomorphic automorphic forms}
	
	We finally show the following classification:
	
	\begin{thm}\label{classification_th}
	Let $M$ be an indecomposable reducible $(\frakg_n,K_{n,\infty})$-module generated by a nearly holomorphic modular form.
	Then, the length of $M$ is at most two.
	Moreover, if $F \neq \bbQ$, $M$ is irreducible.
	If $F=\bbQ$ and $M$ is reducible, let $L(a_1,\ldots,a_n)$ be the socle of $M$ and $L(b_1,\ldots,b_n)$ the irreducible quotient of $M$.
	Then, there exists $i$ such that
	\begin{itemize}
	\item $a_j = b_j$ for $j = 1,\ldots,n-i$.
	\item $a_{n-i+1}=\cdots=a_{n} = \rho_{i,n}-1$ and $b_{n-i+1}=\cdots=b_{n} = \rho_{i,n}+1$.
	\item $M \cong N(a_1,\ldots,a_n)^\vee$.
	\end{itemize}
	Moreover, if a reducible module $M$ has a regular infinitesimal character, one has $i=1$.
	\end{thm}
	\begin{proof}
	We may assume $M$ is reducible.
	There exists $s_0 \in (1/2)\bbZ_{\geq 0}$, a positive integer $i$, a character $\mu$ of $\GL_i(\bbA_F)$ and an irreducible cuspidal automorphic representation $\pi$ of $G_{n-i}(\bbA_\bbQ)$ such that the indecomposable reducible module $M$ can be embedded into $\calN(G_{n},\chi_{s_0})_{(M_{Q_{i,n}},\mu\boxtimes\pi)}$.
	By Lemma \ref{isotypic_s=0}, Proposition \ref{F neq Q, P = B}, Proposition \ref{F = Q, P = B}, Proposition \ref{general_P_neq_B} and Proposition \ref{P neq B}, since $M$ is reducible, one has $F=\bbQ$ and $s_0 = 1$.
	In the following, we assume $F=\bbQ$ and $s_0 = 1$.
	
	Put
	\[
	M_1 = M \cap \calN^2(G_{n},\chi_{s_0})_{(M_{Q_{i,n}},\Pi)}.
	\]
	Then, the submodule $M_1$ is semisimple.
	Since the submodule $M_1$ occurs in 
	\[
	I_{i,n}(-1,\mu,\pi)_{\frakp_{n,-}\fin},
	\]
	the module $M_1$ is isomorphic to $L(\lambda_1,\ldots,\lambda_{n-i}, \rho_{i,n}-1,\ldots,\rho_{i,n}-1)$ with some multiplicities.
	Put $M_2 = M/ M_1$.
	Then, one obtains that $M_2$ is isomorphic to $L(\lambda_1,\ldots,\lambda_{n-i}, \rho_{i,n}+1,\ldots,\rho_{n,i}+1)$ with some multiplicities by Proposition \ref{F = Q, P = B} (2) and Proposition \ref{P neq B}.
	By Lemma \ref{Ext}, the module $M$ is isomorphic to $N(\lambda_1,\ldots,\lambda_{n-i}, \rho_{i,n}-1,\ldots,\rho_{n,i}-1)^\vee$.
	
	If $M$ has a regular infinitesimal character, the socle $L(a_1,\ldots,a_n)$ has a regular infinitesimal character.
	Then, one has $i=1$.
	This completes the proof.
	\end{proof}
	
	\begin{rem}
	A typical example of nearly holomorphic modular form that generates an indecomposable reducible module is $E_2$.
	Here, $E_2$ is defined by
	\[
	E_2(z) = \frac{3}{\pi y} -1 + 24 \sum_{n=1}^\infty \left(\sum_{0 < d | n} d\right) \exp(2\pi \sqrt{-1}nz), \qquad z \in \frakH_1.
	\]
	Then, $E_2$ generates $N(0)^\vee$.
	For details, see \cite{Horinaga_1}.
	\end{rem}
	
	\begin{cor}\label{classif_with_K-type}
	Let $\lambda$ be a regular anti-dominant integral weight and $\chi = \chi_\lambda$.
	Let $\calN\mathrm{Rep}_n(\chi)$ be the set of isomorphism classes of indecomposable $(\frakg_n,K_{n,\infty})$-modules with the regular integral infinitesimal character $\chi$ generated by nearly holomorphic Siegel modular forms of degree $n$.
	For a $K_{n,\infty}$-type $\sigma$, put $\calN\mathrm{Rep}_n(\chi,\sigma) = \{\pi \in \calN\mathrm{Rep}_n(\chi) \mid \text{$\pi$ has the $K_{n,\infty}$-type $\sigma$}\}$.
	We then have
	\[
	\calN\mathrm{Rep}_n(\chi) \subset 
	\begin{cases}
	\{L(\lambda^{(0)}),\ldots,L(\lambda^{(p)}), N(\lambda^{(1)})^\vee\} & \text{if $\lambda_{n} = n+1$}\\
	\{L(\lambda)\}&\text{if $\lambda_n \neq n+1$}
	\end{cases}
	\]
	and
	\[
	\calN\mathrm{Rep}_n(\chi, \mathrm{det}^{\lambda_1-1}\otimes\wedge^{j(\lambda)}) \subset
	\begin{cases}
	 \{L(\lambda^{(0)}), N(\lambda^{(1)})^\vee\}& \text{if $\lambda_{n} = n+1$}\\
	\{L(\lambda)\}&\text{if $\lambda_n \neq n+1$}.
	\end{cases}
	\]
	\end{cor}
	\begin{proof}
	The statement follows immediately from Proposition \ref{distinguish_mod} and Theorem \ref{classification_th}.
	\end{proof}

\section{Projection operators}
	
	In this section, we investigate projection operators associated to infinitesimal characters.
		
	\subsection{Generators of $\calZ_n$}
	
	In this subsection, we assume $F = \bbQ$ for simplicity.
	It is well-known that $\calZ_n$ is generated by $n$ generators.
	We give generators explicitly.
	We first define matrices $B_{i,j}$ and $E_{\pm,i,j} = E_{\pm,j,i}$ as follows:
	\[
	B_{i,j} = \begin{pmatrix} \frac{1}{2} (e_{i,j} - e_{j,i}) & \frac{-\sqrt{-1}}{2}(e_{i,j} + e_{j,i}) \\ \frac{\sqrt{-1}}{2}(e_{i,j} + e_{j,i}) & \frac{1}{2}(e_{i,j} - e_{j,i})\end{pmatrix}, 
	\qquad E_{\pm,i,j} = \begin{pmatrix} \frac{1}{2} (e_{i,j} + e_{j,i}) & \frac{\pm\sqrt{-1}}{2}(e_{i,j} + e_{j,i}) \\ \frac{\pm\sqrt{-1}}{2}(e_{i,j} + e_{j,i}) & \frac{-1}{2}(e_{i,j} + e_{j,i})\end{pmatrix}
	\]
	Then, $\{B_{i,j} \mid  1 \leq i,j \leq n\}$ and $\{E_{\pm,i,j} \mid 1 \leq i \leq j \leq n\}$ are basis of $\frakk_n$ and $\frakp_{n,\pm}$, respectively.
	We define $B \in \Mat_n(\Mat_{2n}(\bbC))$ and $E_\pm \in \Sym_n(\Mat_{2n}(\bbC))$ by
	\[
	B = (B_{k,\ell})_{k,\ell}, \qquad E_\pm = (E_{\pm,k,\ell})_{k,\ell} \in \Sym_n(\Mat_{2n}(\bbC)).
	\]
	Put $B^* = (B_{j,i})_{i,j}$, the transpose of $B$.
	Let $w = X_1 \cdots X_m$ be a word with letters $B,B^*$ and $E_{\pm}$.
	We assume the word $w$ satisfies the following five conditions:
	\begin{itemize}
	\item $E_+$ is followed by $E_-$ or $B^*$.
	\item $E_-$ is followed by $E_+$ or $B$.
	\item $B$ is followed by $E_+$ or $B$.
	\item $B^*$ is followed by $E_-$ or $B^*$.
	\item $E_+$ and $E_-$ occur with the same multiplicity.
	\end{itemize}
	For a word, let $\tr(w) \in \Mat_{2n}(\bbC)$ be the trace as the $\Mat_{2n}(\bbC)$-valued matrix.
	We may identify $\tr(w)$ as an element of $\calU(\frakg_n)$.
	Let $L(w)$ be the sum of number of times $E_-B$ and $BE_+$ occur isolatedly in $w$ counted cyclicly.
	For example, $L(E_-BE_+) = 0, L(E_-BE_+B^*) = 1, L(E_+E_-BB) = L(E_-BBE_+) = 2$.
	Put
	\[
	D_{2r} = \sum_{w} (-1)^{L(w)}\tr(w)
	\]
	where $w$ runs over all words of length $2r$ with the above five conditions.
	
	\begin{thm}[\cite{2012_Maurischat}]
	The algebra $\calZ_n$ is generated by elements $D_{2}, \ldots, D_{2n}$ as an algebra over $\bbC$.
	\end{thm}
	
%
%
%
%
%

\subsection{Projection operators}
	
	Fix an infinitesimal character $\chi$, a weight $\rho$ and a congruence subgroup $\Gamma$.
	Let $K_\Gamma$ be the closure of $\Gamma$ in $G_{n}(\bbA_{\bbQ,\fini})$.
	We now define a projection on $\calN(G_n)_\rho^{K_\Gamma}$.
	Let $\lambda$ be the highest weight of $\rho$.
	We define a set $X(\rho)$ of $\frakk_n$-dominant weights by the set of $\frakk_n$-dominant weights $\mu$ such that $\mu$ satisfies the following three conditions:
	\begin{itemize}
	\item $L(\mu)$ is unitarizable.
	\item $L(\mu)$ has the $K_{n,\infty}$-type $\rho$.
	\item $\lambda \leq \mu$. 
	\end{itemize}
	Then, $X(\rho)$ is finite.
	Put $\chi(\rho) = \{\chi_\mu \mid \mu \in X(\rho)\}$.
	For infinitesimal characters $\chi$ and $\omega$, we define $D_{\chi,\omega}$ as follows:
	Let $v \in \bfa$.
	If the local components $\chi_v$ and $\omega_v$ are the same, put $D_{\chi,\omega,v} = 1$.
	If $\chi_v \neq \omega_v$, there exists $i$ such that $\chi_v(D_{2i}) \neq \omega_v(D_{2i})$.
	Then, put $D_{\chi,\omega,v} = D_{2i} - \omega_v(D_{2i})$.
	Set $D_{\chi,\omega} = \bigotimes_{v \in \bfa} D_{\chi,\omega,v}$.
	By definition, for an $\omega$-eigenvector $v$ with $\omega \in \chi(\rho)$, we have
	\[
	\frac{1}{\chi(D_{\chi,\omega})}D_{\chi,\omega} \cdot v = 
	\begin{cases}
	v &\text{if $\chi = \omega$}\\
	0 &\text{if $\chi \neq \omega$}.
	\end{cases}
	\]
	We now can define the projection $\frakp_\chi \in \End_{\bbC}(\calN(G_n)_\rho^{K_\Gamma})$ by
	\[
	\frakp_\chi(f) = \frac{1}{\prod_{\omega \in \chi(\rho)} \chi(D_{\chi,\omega})} \prod_{\omega \in X(\rho)} D_{\chi,\omega} \cdot f.
	\]
	By (\ref{ss_Z_n}), $\frakp_\chi$ defines a projection onto the $\chi$-eigen subspace of $\calN(G_n)_\rho^{K_\Gamma}$ associated to $\chi$.
	
	By Lemma \ref{corresp_MF_AF}, one has
	\[
	N_\rho(\Gamma) \otimes \rho^* \cong \calN(G_n)_\rho^{K_\Gamma}.
	\]
	The projection defines an endomorphism on $N_\rho(\Gamma) \otimes \rho^*$.
	
	\begin{lem}
	The projection $\frakp_\chi$ defines a projection on $N_\rho(\Gamma)$.
	\end{lem}
	\begin{proof}
	We have the map
	\[
	N_\rho(\Gamma) \longrightarrow \mathrm{Hom}_{K_{n,\infty}} (\rho^*, N_\rho(\Gamma) \otimes \rho^*)
	\]
	by
	\[
	f \longmapsto (v \longmapsto f \otimes v).
	\]
	Since it is injective, it is isomorphism by comparing the dimensions.
	We identify $N_\rho(\Gamma) \otimes \rho^*$ as $\calN(G_n)_\rho^{K_\Gamma}$.
	Let $(N_\rho(\Gamma) \otimes \rho^*)_\chi$ be the $\chi$-eigen subspace of $N_\rho(\Gamma) \otimes \rho^*$ associated to an infinitesimal character $\chi$. 
	Since the $\chi$-isotypic component of $\calN(G_n)_{\rho}^{K_\Gamma}$ is $K_{n,\infty}$-stable, the corresponding space $(N_\rho(\Gamma) \otimes \rho^*)_\chi$ is $K_{n,\infty}$-stable.
	Thus we can define the subspace
	\[
	\mathrm{Hom}_{K_{n,\infty}} (\rho^*, (N_\rho(\Gamma) \otimes \rho^*)_\chi)
	\]
	of $\mathrm{Hom}_{K_{n,\infty}} (\rho^*, N_\rho(\Gamma) \otimes \rho^*)$ and of $N_\rho(\Gamma)$.
	We denote the subspace of $N_\rho(\Gamma)$ by $N_\rho(\Gamma,\chi)$.
	Since $\calN(G_n)_\rho^{K_\Gamma}$ decomposes as the direct sum of $\chi$-eigen spaces, one has $N_\rho(\Gamma) = \bigoplus_\chi N_\rho(\Gamma,\chi)$.
	By the map $F \longmapsto \frakp_\chi \circ F$, one obtains a map $\mathrm{Hom}_{K_{n,\infty}} (\rho^*, N_\rho(\Gamma) \otimes \rho^*) \longrightarrow \mathrm{Hom}_{K_{n,\infty}} (\rho^*, (N_\rho(\Gamma) \otimes \rho^*)_\chi)$ and thus it induces the map $N_\rho(\Gamma) \longrightarrow N_\rho(\Gamma,\chi)$.
	It suffices to show that this map is a projection.
	For $f \in N_\rho(\Gamma,\chi)$, one can regard $f$ as an element $F \in \mathrm{Hom}_{K_{n,\infty}} (\rho^*, (N_\rho(\Gamma) \otimes \rho^*)_\chi)$.
	Since $\frakp_\chi$ is projection, one has $\frakp_\chi \circ F = F$.
	It shows that $f$ is invariant under the map $N_\rho(\Gamma) \longrightarrow N_\rho(\Gamma,\chi)$.
	Thus, this map is an idempotent and hence a projection.
	This completes the proof.
	\end{proof}
	
	We denote by the same letter $\frakp_\chi$ the projection on $N_\rho(\Gamma)$ as in the above lemma.
	Thus we have $\frakp_\chi(f \otimes v^*) = \frakp_\chi(f) \otimes v^*$ for $f \in N_\rho(\Gamma)$ and $v^* \in \rho^*$.
	Set $N_\rho(\Gamma,\chi) = \frakp_\chi(N_\rho(\Gamma))$.
	
	\begin{thm}\label{proj_comm}
	The projection $\frakp_\chi$ on $N_\rho(\Gamma)$ commutes with the $\Aut(\bbC)$-action.
	\end{thm}
	\begin{proof}
	The case where $F=\bbQ$ is proved in \cite[Proposition 3.16]{HPSS}.
	The general case is similar.
	We omit the details.
	\end{proof}
	
	For an integral weight $\lambda = (\lambda_{1,v},\ldots,\lambda_{n,v})_v$, put $j_v(\lambda) = \#\{i \mid \text{$\lambda_{1,v} \equiv \lambda_{i,v}$ mod $2$}\}$.
	
	\begin{thm}\label{proj}
	Let $\lambda = (\lambda_{1,v},\ldots,\lambda_{n,v})_v$ be a regular anti-dominant $\frakk_n$-dominant integral weight.
	Put $\rho = \bigotimes_{v \in \bfa} (\det^{\lambda_{1,v}-1} \otimes \wedge^{j_v(\lambda)})$.
	If $F=\bbQ$ and $\lambda_{n,v} = n+1$, any modular form in $N_\rho(\Gamma,\chi_\lambda)$ generates $L(\lambda)$ or $N(\lambda^{(1)})^\vee$.
	If not, any modular form in $N_\rho(\Gamma,\chi_\lambda)$ generates $L(\lambda)$.
	\end{thm}
	\begin{proof}
	Take $f \in N_\rho(\Gamma,\chi_\lambda)$.
	Then the $(\frakg_n,K_{n,\infty})$-module generated by $f$ is a direct sum of modules in $\calN \mathrm{Rep}_n(\chi_\lambda,\rho)$.
	Thus, the statement follows from Corollary \ref{classif_with_K-type}.
	\end{proof}
	
	We finally give an analogue of holomorphic projections.
	
	\begin{cor}
	Let $\lambda = (\lambda_{1,v},\ldots,\lambda_{n,v})_v$ be a regular anti-dominant integral weight with $\lambda_{1,v}-\lambda_{n,v} \leq 1$ for any $v \in \bfa$ and $\rho$ the irreducible highest weight representation of $K_{n,\bbC}$ with highest weight $\lambda$.
	If $F \neq \bbQ$ or $\lambda_{n,v} \neq n+1$ for some $v \in \bfa$, the projection $\frakp_\chi$ defines a projection onto $M_\rho(\Gamma)$.
	\end{cor}
	\begin{proof}
	By $\lambda_{1,v} - \lambda_{n,v} \leq 1$ and Theorem \ref{proj}, any modular form $f$ in $N_\rho(\Gamma,\chi_\lambda)$ generates $L(\lambda)$.
	Since $f$ is of weight $\lambda$, $f$ corresponds to a highest weight vector.
	Thus, $f$ is holomorphic and $N_\rho(\Gamma,\chi_\lambda) = M_\rho(\Gamma)$.
	This completes the proof.
	\end{proof}

\bibliographystyle{alpha}
\bibliography{ref}

\end{document}